%% file: unitDistGraph.tex
\documentclass[11pt]{amsart}
\usepackage{fullpage}
\usepackage{amsmath, amssymb, amsthm}
\usepackage[foot]{amsaddr}
\usepackage[inline]{enumitem}

\usepackage{tikz}
\usetikzlibrary{graphs,graphs.standard}

\usepackage{wrapfig}

\usepackage{graphicx}
\usepackage{color}
\usepackage[colorlinks=true,linkcolor=red,citecolor=blue,urlcolor=cyan]{hyperref}
\usepackage{cleveref}

\newtheorem{thm}{Theorem}[section]
\newtheorem{cor}[thm]{Corollary}
\theoremstyle{remark}
\newtheorem{defin}[thm]{Definition}
\newtheorem{exam}[thm]{Example}

\newcommand{\ie}{{\it i.e.}} 
\newcommand{\eg}{{\it e.g.}}

\usepackage[backend=biber,giveninits=true,sorting=nyt,style=alphabetic,natbib=true,maxcitenames=2,maxbibnames=10,url=false,doi=true,backref=false,date=year]{biblatex}
\renewbibmacro{in:}{\ifentrytype{article}{}{\printtext{\bibstring{in}\intitlepunct}}}

\graphicspath{{../figs/}} 

\begin{document}
\title{Extremal graph realizations and \\ graph Laplacian eigenvalues}

\author{Braxton Osting}
\address{Department of Mathematics, University of Utah, Salt Lake City, UT}
\email{osting@math.utah.edu}
\thanks{Braxton Osting acknowledges partial support from NSF DMS 17-52202.}


\date{\today}

\begin{abstract} 
For a regular polyhedron (or polygon) centered at the origin, the coordinates of the vertices are eigenvectors of the graph Laplacian for the skeleton of that polyhedron (or polygon) associated with the first (non-trivial) eigenvalue. In this paper, we generalize this relationship. For a given graph, we study the eigenvalue optimization problem of maximizing the first (non-trivial) eigenvalue of the graph Laplacian over non-negative edge weights. We show that the spectral realization of the graph using the eigenvectors corresponding to the solution of this problem, under certain assumptions, is a centered, unit-distance graph realization that has maximal total variance.  This result gives a new method for generating unit-distance graph realizations and is based on convex duality. A drawback of this method is that the dimension of the realization is given by the multiplicity of the extremal eigenvalue, which is typically unknown prior to solving the eigenvalue optimization problem. Our results are illustrated with a number of examples. 
\end{abstract}


\maketitle

\section{Introduction} \label{s:Intro}
There is a beautiful observation attributed to C.~D. Godsil that the coordinates of the vertices of some polytopes are the eigenvectors of the adjacency matrix for the skeleton of that polytope \cite{Godsil_1978}. For regular polygons and polyhedra that are centered at the origin, this connection is not difficult to see geometrically as, by symmetry,  averaging the neighboring vertices  $N(i) \subset V$ of a given vertex $i \in V$ gives a vector that is collinear with the vertex, \ie, there exists $\alpha \in \mathbb R$ such that
$$
\sum_{j \in N(i)} x_j  = \alpha  x_i,   \qquad  i \in V.
$$ 
See \cref{f:GraphRealizations}(a) for an illustration. 
This connection was established for platonic solids and some other regular polytopes in \cite{Powers_1988,LicataPowers}. 
This observation can be used to motivate the use of ``spectral embeddings'' in data analysis, where a graph is ``embedded'' in $d$-dimensional Euclidean space using the first $d$ (non-trivial)  eigenvectors of the adjacency matrix or another matrix associated with the graph, such as the graph Laplacian. 

In this paper, we extend this result by investigating the relationship between graph realizations and the eigenvectors for certain \emph{weighted} graph Laplacians. 

\subsection*{Graph realizations}
Let $G = (V,E)$ be a connected, undirected graph with $n=|V|$ vertices and $m=|E|$ edges. 
We enumerate the vertices and edges and, when convenient, identify $V = [n]:= \{1,\ldots,n\}$ and $E=[m]$. 
For fixed $d\in \mathbb N$, a  
\emph{$d$-dimensional graph realization}\footnote{Note that a graph realization is not a topological embedding as it might not be injective. See \cref{f:GraphRealizations}(c).}  
of $G$ is a mapping $x \colon V \to \mathbb R^d$ with \emph{coordinate matrix} $X = [ x(1) \mid x(2) \mid \cdots \mid x(n)]^t \in \mathbb R^{n \times d}$. 
A graph realization is \emph{centered} if $X^t 1 = 0$. 
A $d$-dimensional graph realization is \emph{unit-distance} if  $\|x(i) - x(j)\| = 1$ for every $(i,j) \in E$. 
Some examples of unit-distance graph realizations are given in \cref{f:GraphRealizations}.
Let $B = \left[b_1 \mid  b_2 \mid \cdots \mid b_m \right]^t \in \mathbb R^{m \times n}$ 
be an associated (arc-vertex) \emph{incidence matrix}, given by 
$$
B_{k,i} = 
\begin{cases} 
1 & i = \textrm{head}(k) \\
- 1 & i = \textrm{tail}(k) \\
0 & \textrm{otherwise}
\end{cases},
$$
where the orientation of the edges are arbitrarily chosen. 
This unit-distance condition for a graph realization can be equivalently expressed in terms of the coordinate and incidence matrices by 
$$
\| X^t b_k \| = 1,  \qquad \forall k \in E. 
$$
Note that for a graph $G$, depending on the dimension $d$, there may: 
\begin{enumerate}[label=(\roman*)]
\item not exist a centered unit-distance graph realization (\eg, $d=1$ and the complete graph on 3 vertices), 
\item exist a unique (up to rotation) centered unit-distance graph realization (\eg, $d\geq1$ and the complete graph on $d+1$ vertices; see \cref{f:GraphRealizations}(d)), or 
\item there may exist a family of unit-distance graph realizations (\eg, $d=2$ and a centered polygon with $\geq 4$ vertices and unit-edge lengths; see \cref{f:GraphRealizations}(b).)  
\end{enumerate}  

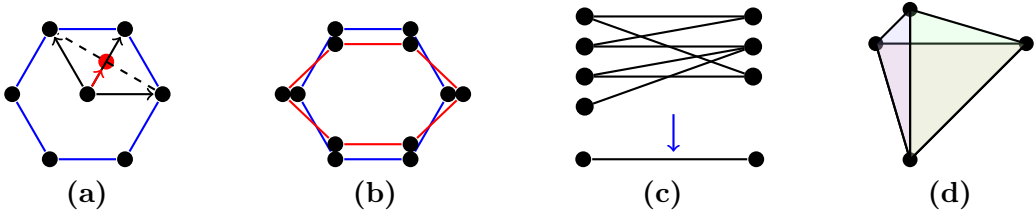
\begin{figure}[t!]
\begin{center}
\input{./tikz/Illustration1.tex} \hspace{1cm}
\input{./tikz/Illustration2.tex} \hspace{1cm}
\input{./tikz/Illustration3.tex} \hspace{1cm}
\input{./tikz/Illustration4.tex} \\
{\bf (a)} \hspace{3cm} {\bf (b)} \hspace{3cm} {\bf (c)} \hspace{3cm} {\bf (d)}
\caption{
{\bf (a)} For any vertex of a regular polygon or polyhedron, averaging the adjacent vertices gives a collinear vector.  
{\bf (b)} Two centered unit-distance realizations of the 6-cycle graph. The regular hexagon is the solution to \eqref{e:MaxGraphRealization}. 
{\bf (c)} A two point graph realization of a bipartite graph.
{\bf (d)} A tetrahedron is a 3-dimensional unit-distance graph realization of the 4-vertex complete graph. }
\label{f:GraphRealizations}
\end{center}
\end{figure}

In case (iii) above, when the unit-distance graph realization is not unique, we may  consider the realization with maximal \emph{total variance}, 
$ \| X \|_F^2 = \sum_{i\in V} \|x(i) \|^2.$
Thus, for fixed graph $G$ and dimension $d$, we consider the nonlinear optimization problem of finding the centered unit-distance graph realization with maximum total variance,
\begin{subequations}
\label{e:MaxGraphRealization}
\begin{align}
\max_{X \in \mathbb R^{n \times d}} \ & \| X \|_F^2 \\
\textrm{s.t.} \ & \| X^t b_k \|^2 = \phi_k,  \qquad \forall k \in E \\ 
& X^t 1 = 0.
\end{align}
\end{subequations}
Here, we have slightly generalized the  problem by introducing a vector $\phi \in \mathbb R^m_+ := \{ \phi \in \mathbb R^m \colon \phi_i \geq 0, \ \forall i \in [m]\}$ which specifies the squared edge-lengths; $\phi = 1$ in the case that we seek a unit-distance graph. 
In \cref{f:graphZoo}, we plot a variety of graph realizations corresponding to solutions of \eqref{e:MaxGraphRealization} with $\phi = 1$. We'll explain how we computed these solutions momentarily. 

\begin{figure}[th!]
\begin{center}
\includegraphics[width=0.27\textwidth]{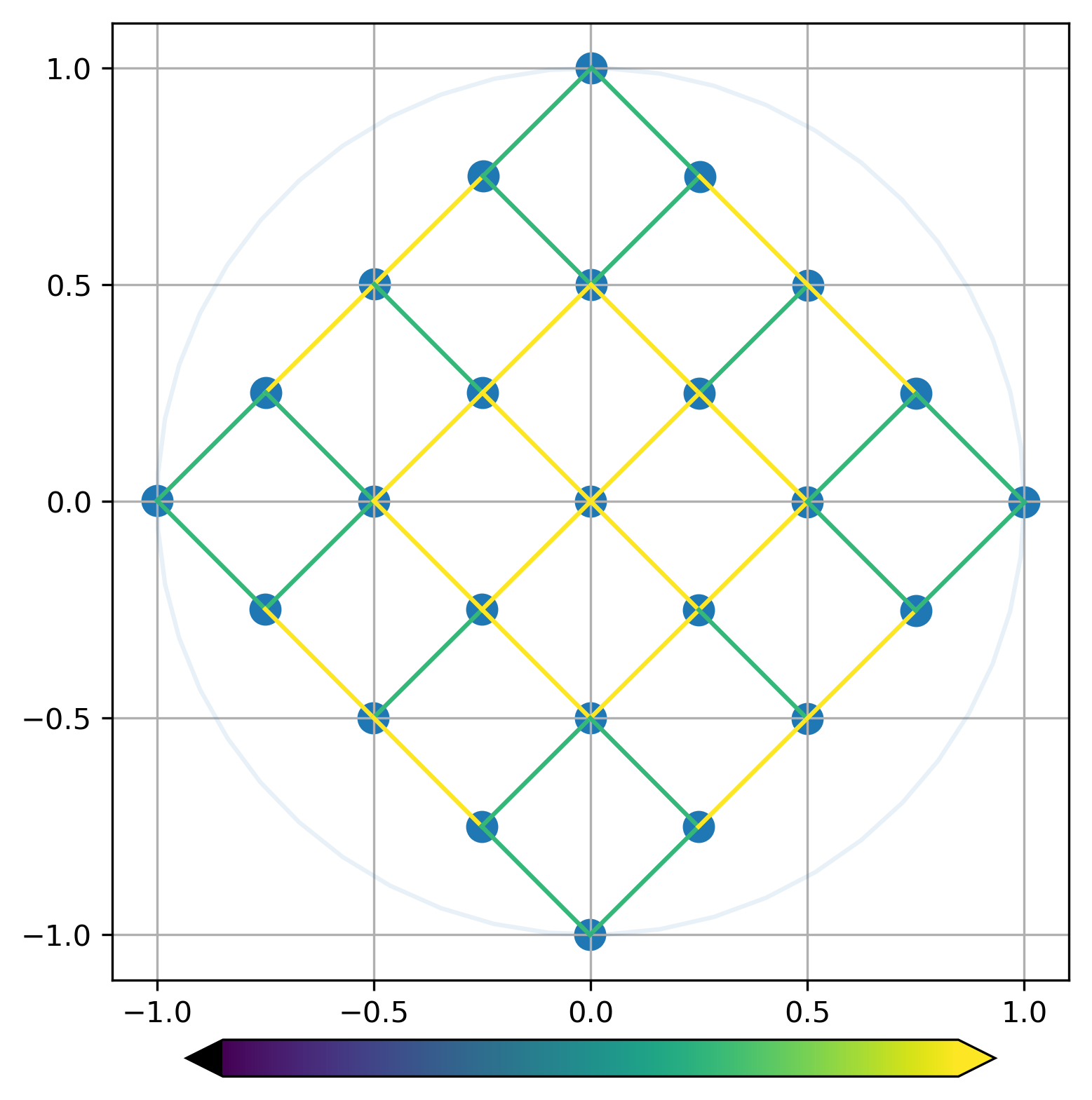}
\includegraphics[width=0.27\textwidth]{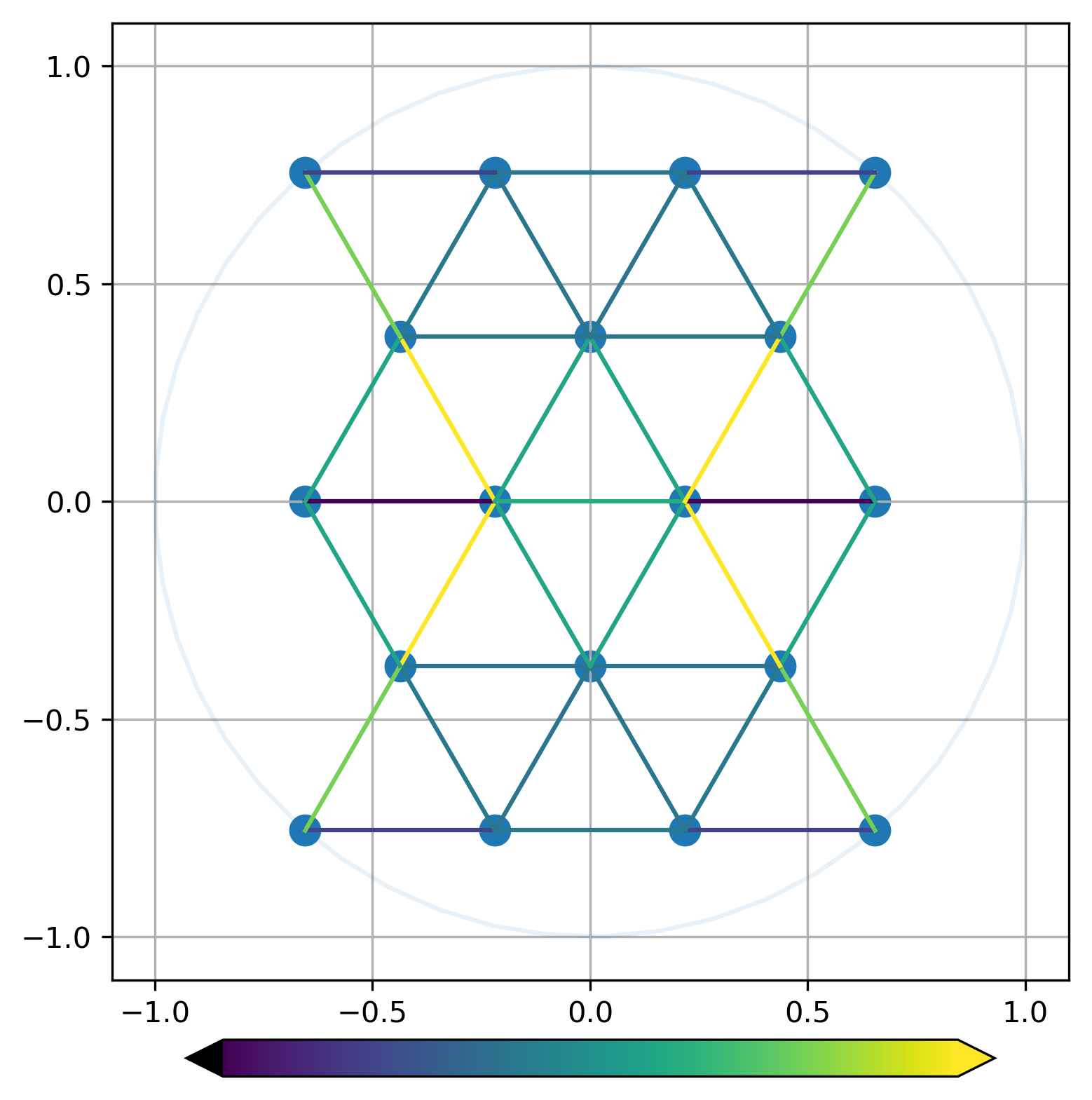}
\includegraphics[width=0.27\textwidth]{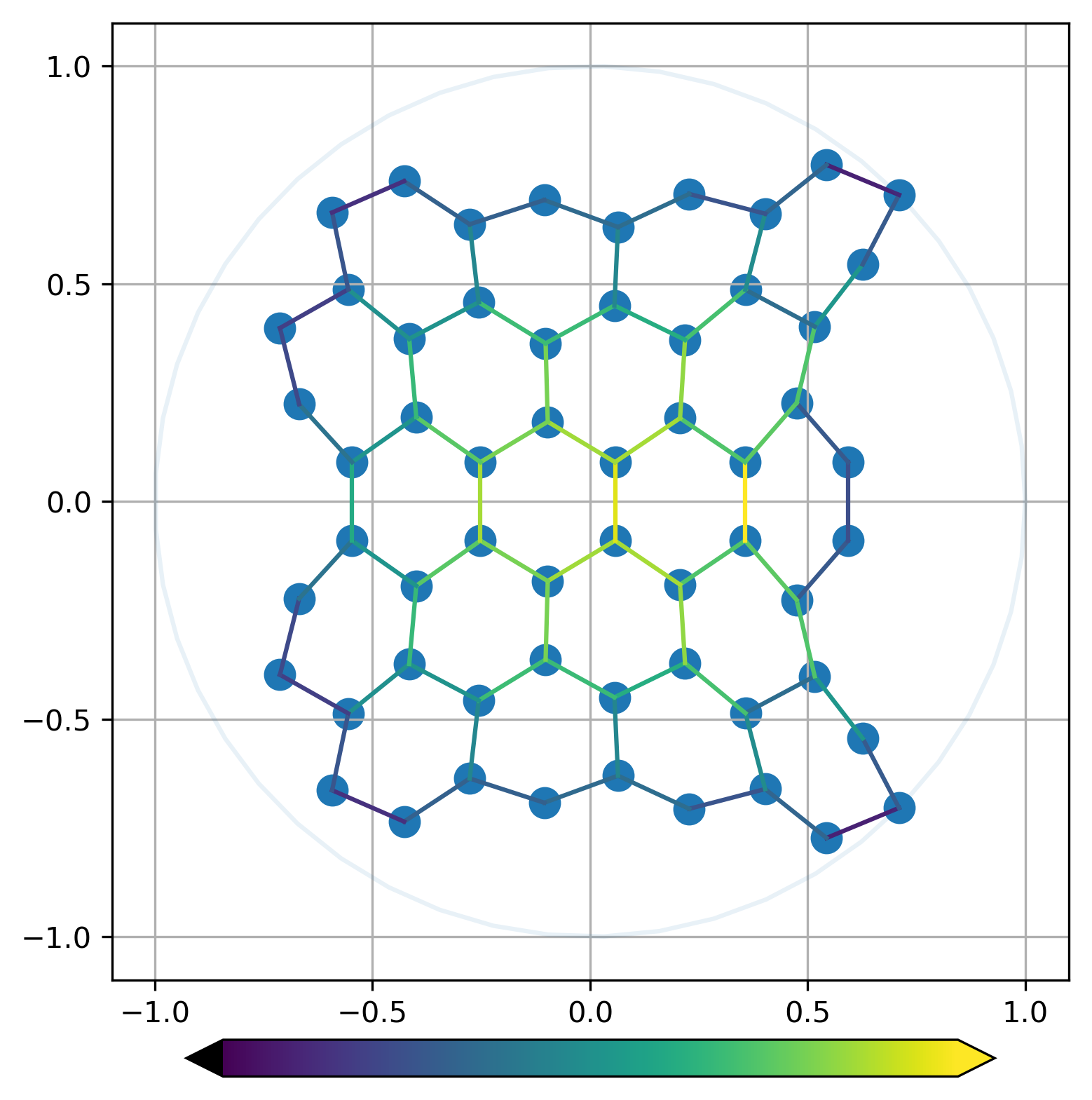} \\
\includegraphics[width=0.27\textwidth]{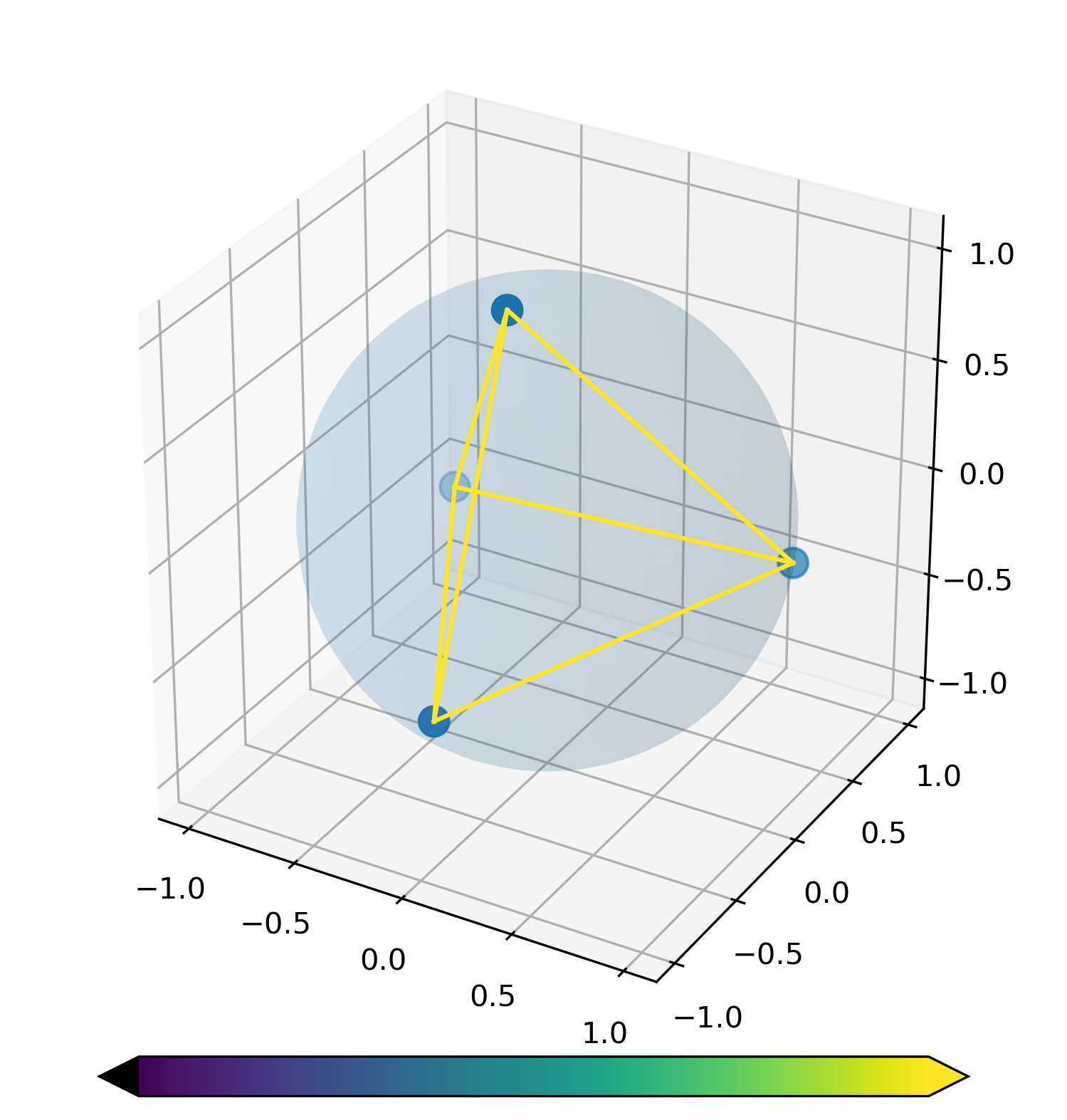}
\includegraphics[width=0.27\textwidth]{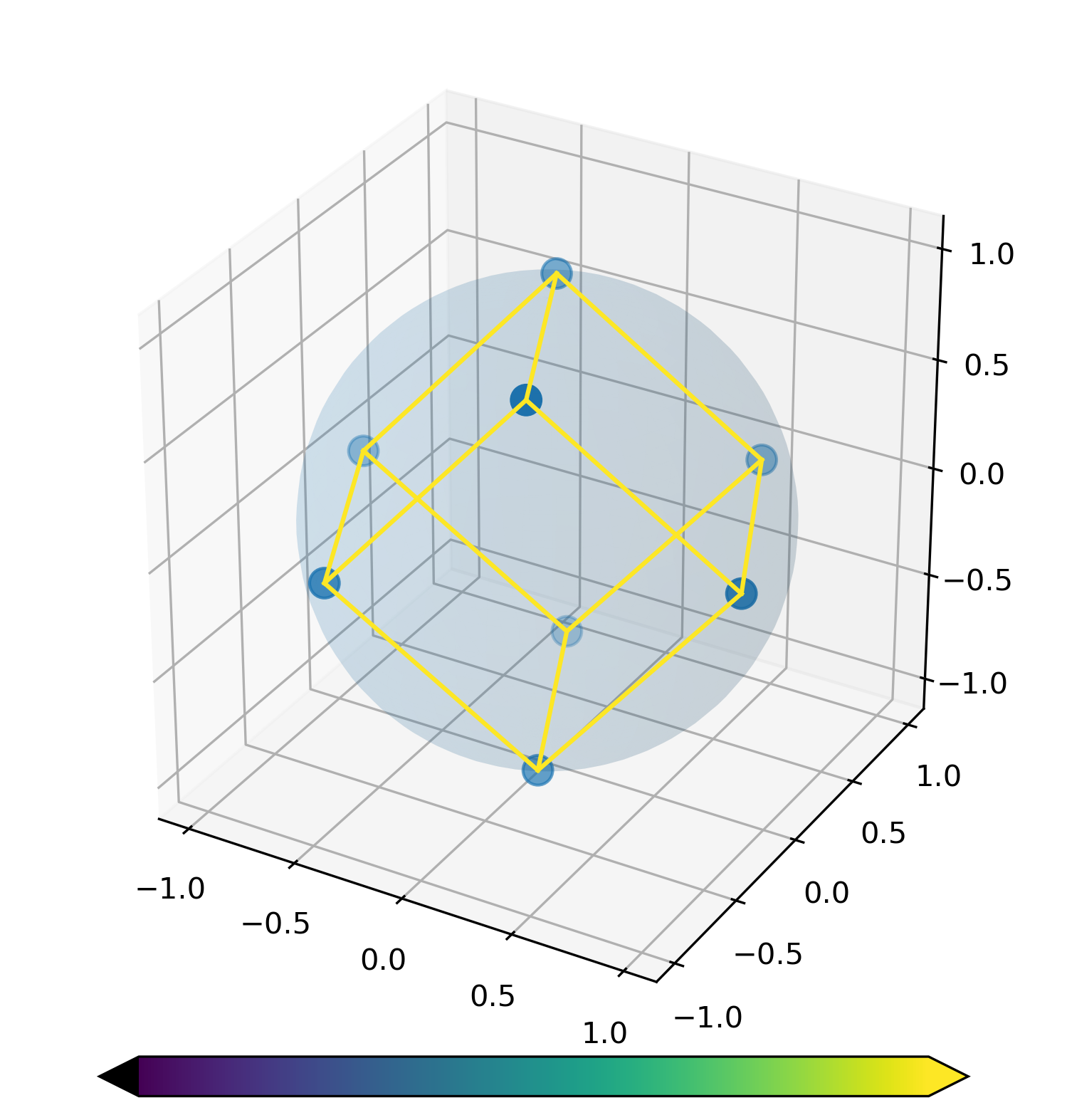}
\includegraphics[width=0.27\textwidth]{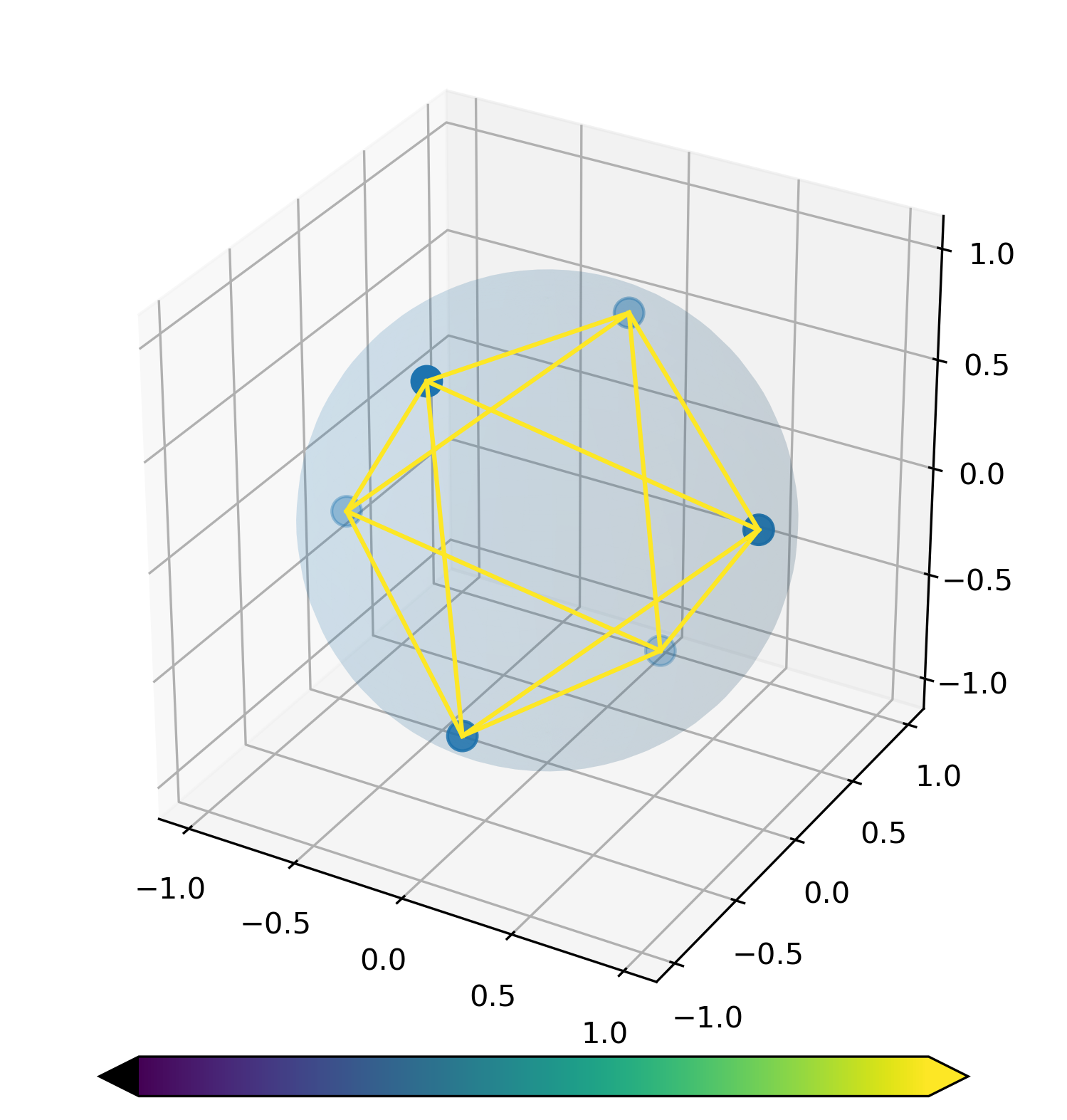} \\
\includegraphics[width=0.27\textwidth]{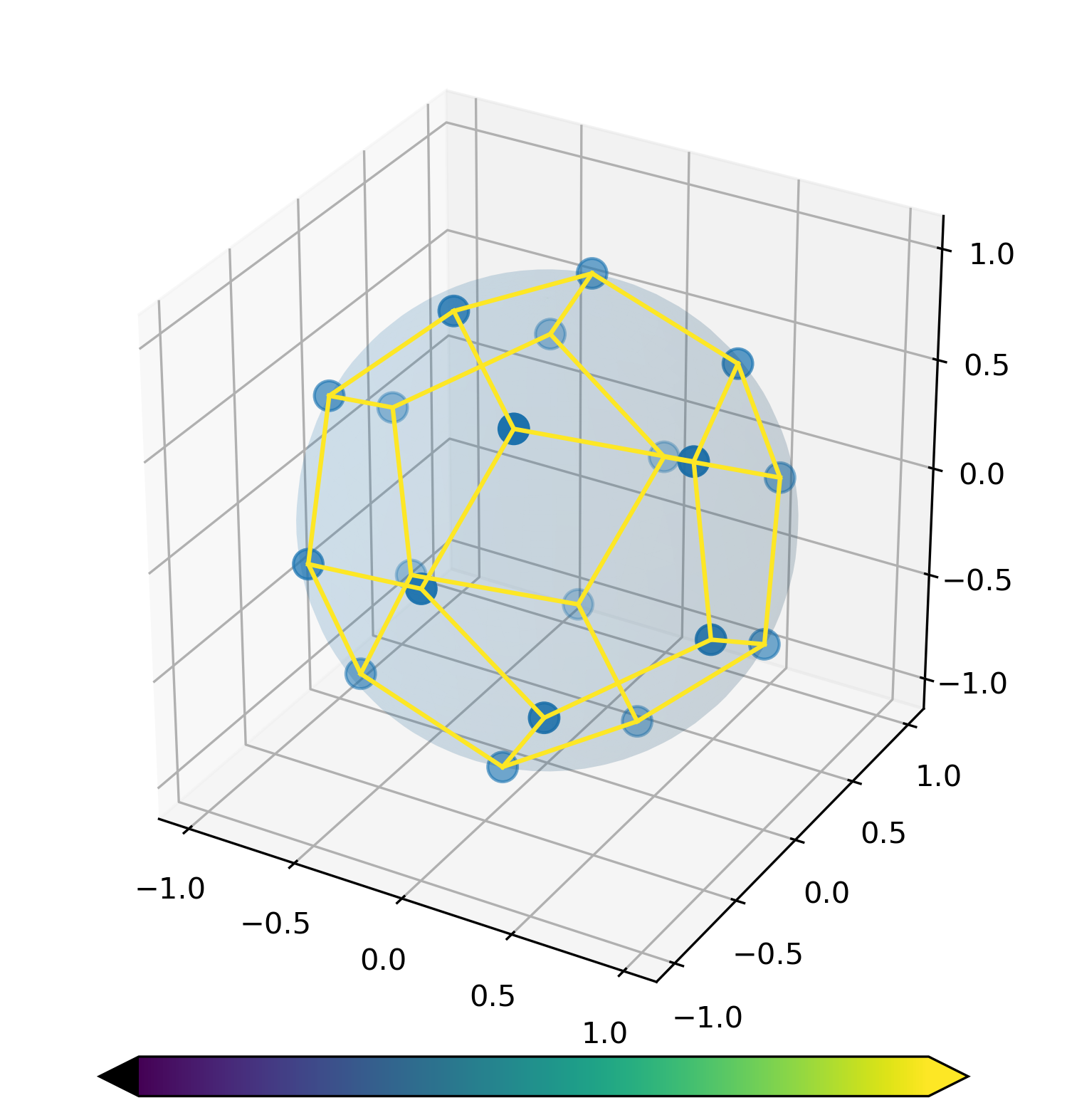} 
\includegraphics[width=0.27\textwidth]{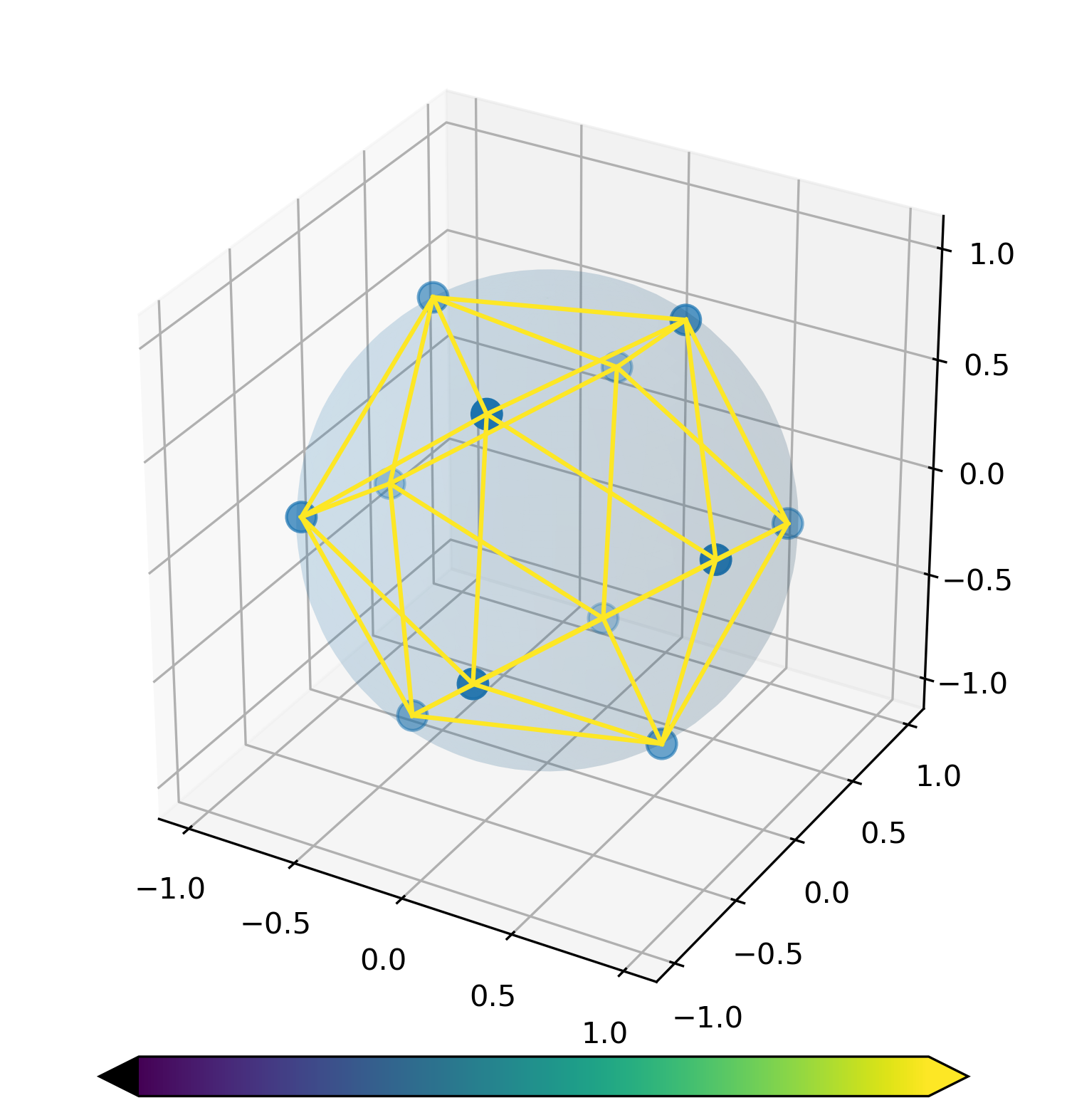}
\includegraphics[width=0.27\textwidth]{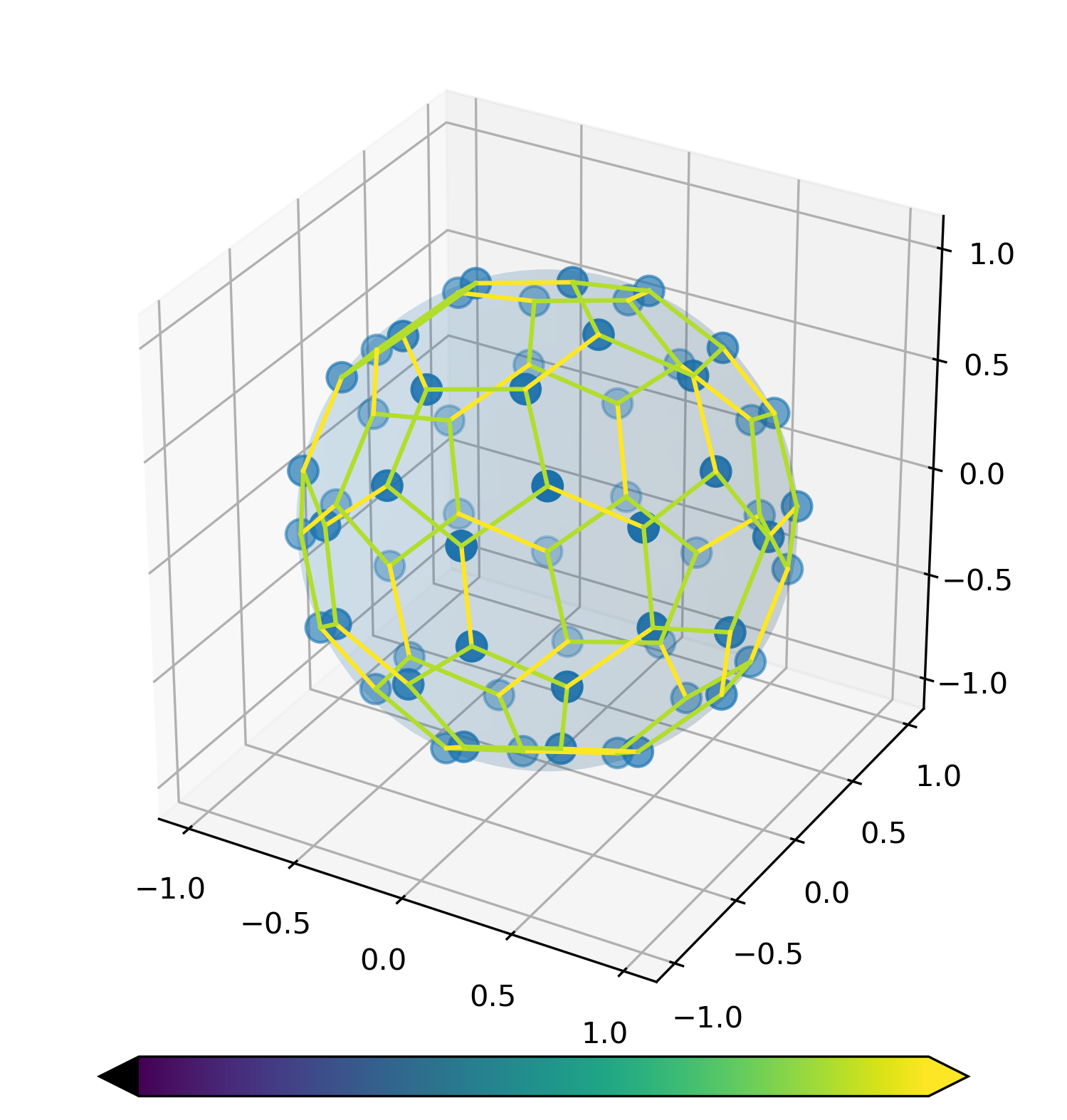} \\
\includegraphics[width=0.27\textwidth]{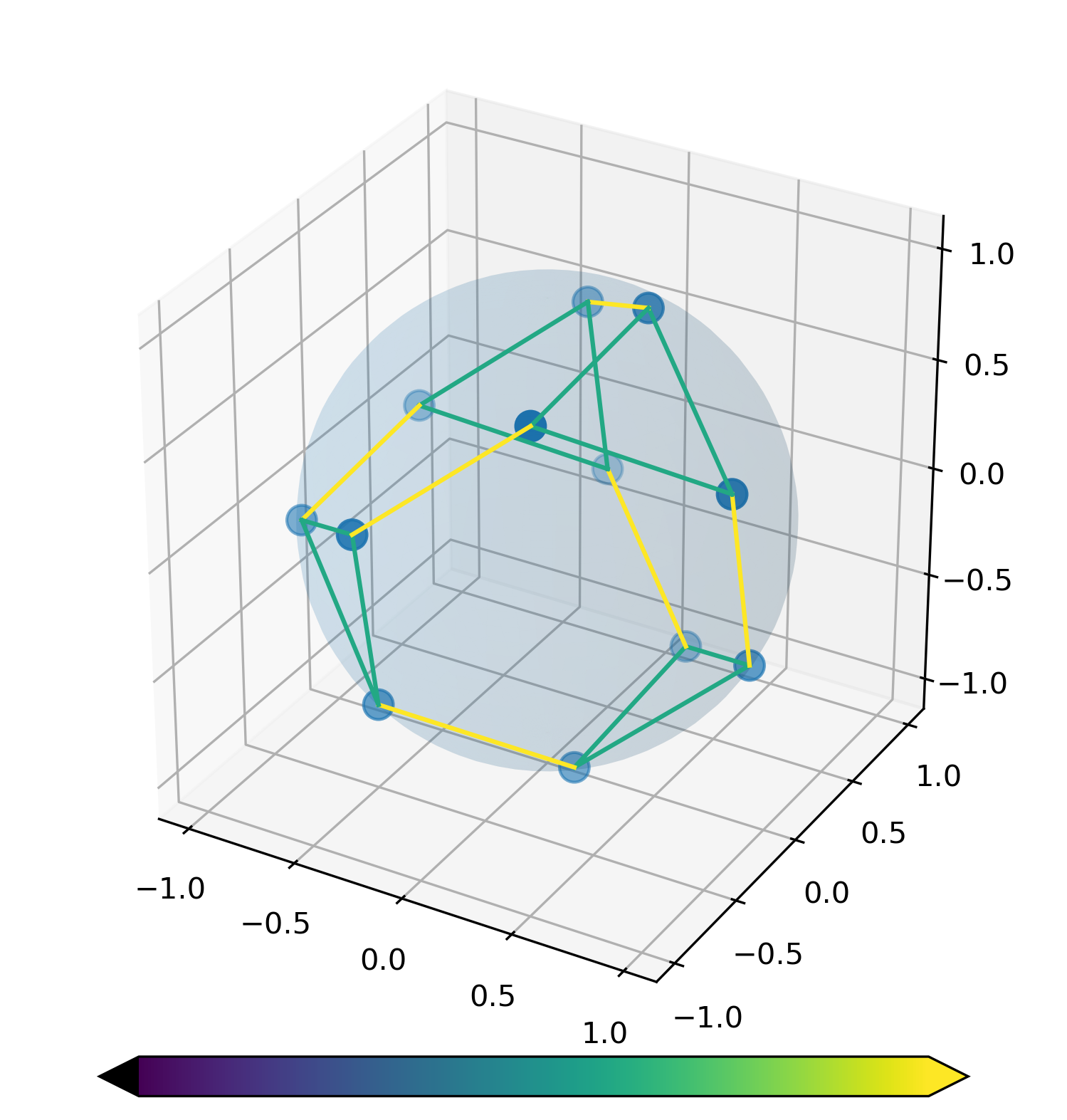}
\includegraphics[width=0.27\textwidth]{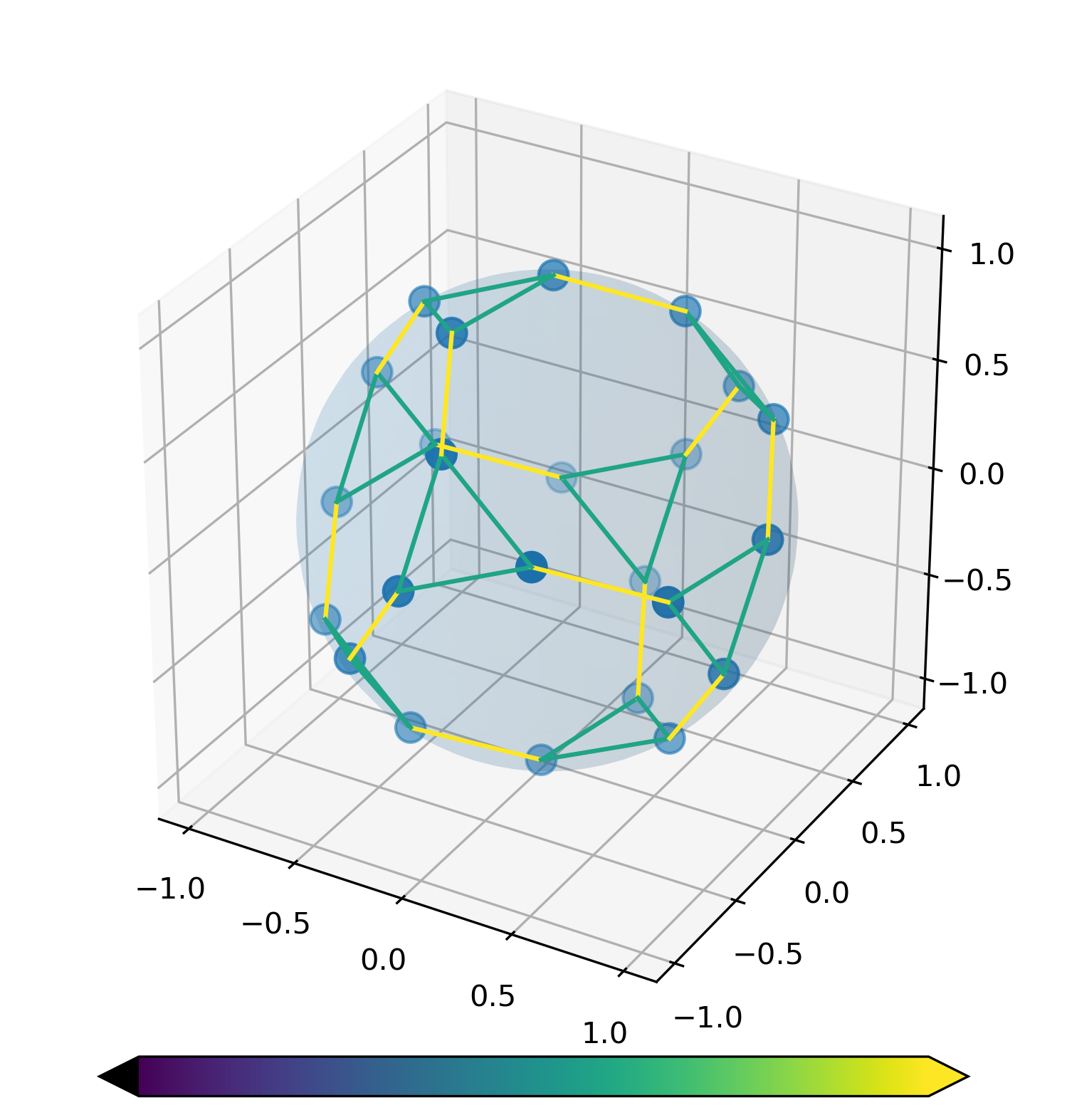}
\includegraphics[width=0.27\textwidth]{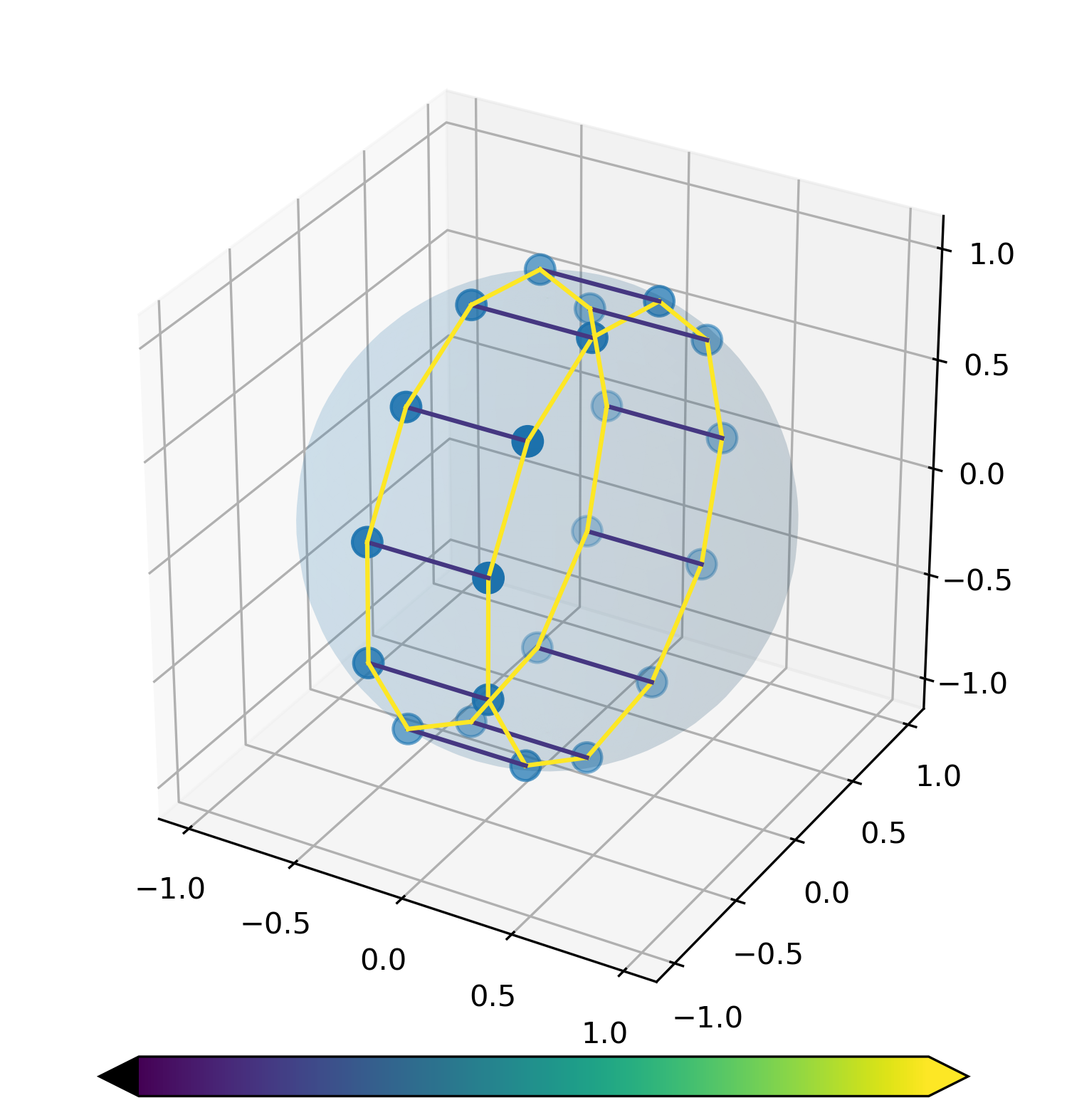}
\caption{Maximal spectral realizations solving \eqref{e:MaxGraphRealization2} for a variety of graphs: 
square lattice graph, 
triangular lattice graph, 
hexagonal lattice graph, 
tetrahedral graph, 
cube graph, 
dodecahedral graph, 
icosahedral graph, 
octahedral graph, 
buckyball graph, 
truncated tetrahedron graph, 
truncated cube graph, and 
circular ladder graph. 
All plots have been scaled so that the largest magnitude vertex has magnitude 1. 
See \cref{s:NumExamples:A} for more details.}
\label{f:graphZoo}
\end{center}
\end{figure}

To accommodate case (i) above, when there does not exist a realization satisfying the edge-length constraint, it will be convenient to relax the constraint and also consider the optimization problem
\begin{subequations}
\label{e:MaxGraphRealization2}
\begin{align}
\max_{X \in \mathbb R^{n \times d}} \ & \| X \|_F^2 \\
\textrm{s.t.} \ & \| X^t b_k \|^2 \leq \phi_k,  \qquad \forall k \in E \\ 
& X^t 1 = 0.
\end{align}
\end{subequations}
Note that the constraint set for \eqref{e:MaxGraphRealization2} is non-empty for any graph $G$ and choice of dimension $d$.  
For a given graph $G$ and $\phi \in \mathbb R^m_+$, we say that the solution $X^\star \in \mathbb R^{n \times d}$ to  \eqref{e:MaxGraphRealization2} is a \emph{maximal $d$-dimensional realization of G} or \emph{maximal graph realization} when the dimension and graph are understood.

\subsection*{The weighted graph Laplacian and spectral graph realizations}
Let us recall that a popular method for constructing a graph realization is to use eigenvectors of a matrix corresponding to the graph, such as the (weighted) graph Laplacian. 
Given an incidence matrix, $B$, for a graph $G=(V,E)$, the \emph{graph Laplacian} is  $\Delta_1 = B^t B \in \mathbb S_+^{n}:=\{A \in \mathbb R^{n \times n} \colon A = A^t \}$\footnote{The graph Laplacian we consider here is referred to as the \emph{unnormalized} graph Laplacian. It can also be written $\Delta_1 = D - A$, where $A$ is the vertex adjacency matrix and $D = \textrm{diag}(d)$ is the diagonal matrix with the degrees $d = A1$ on the diagonal.}. 
For weights $w \in \mathbb R^m_{+}$, the  $w$-\emph{weighted graph Laplacian}, is given by
\begin{equation}
\label{e:graphLap}
\Delta_w = B^t \textrm{diag}(w) B  \in \mathbb S_+^{n}. 
\end{equation}
Spectral properties of $\Delta_w$ have been well-studied \cite{Mohar91thelaplacian,Chung_1996,Biyiko_u_2007}. 
We enumerate the eigenvalues of $\Delta_w$ in increasing order, $0 = \lambda_1( \Delta_w)  \leq \lambda_2( \Delta_w)  \leq \cdots \leq \lambda_n( \Delta_w)$. 
Denote the corresponding normalized eigenvectors by  $\{u_i\}_{i \in [n]}$ (arbitrarily choosing vectors from the eigenspace in the case of eigenvalue multiplicity $\geq 2$) and note that $u_1$ is a constant vector. 
A \emph{$d$-dimensional $w$-spectral graph realization of $G$} is given by the rows of the coordinate matrix 
$$
X = [u_2 \mid \cdots \mid u_{d+1}] \in \mathbb R^{n \times d}.
$$
This spectral graph realization (or a variant thereof) is the first step in spectral clustering \cite{von_Luxburg_2007} and is also commonly used for network visualization \cite{Traud_2009}.

In this paper, we consider how the $w$-spectral graph realization of $G$ depends on the choice of weights $w \in \mathbb R^m_{+}$. 
In particular, for a fixed graph $G=(V,E)$ and $\phi \in \mathbb R^m_+$, we consider the eigenvalue  optimization problem 
\begin{subequations}
\label{e:EOP}
\begin{align}
\max_{w\in \mathbb R^m } \ &  \lambda_2(\Delta_w) \\ 
\textrm{s.t.} \ & \Delta_w u_i = \lambda_i   u _i , \qquad \forall i \in [n] \\ 
& w \geq 0, \quad w^t \phi = 1. 
\end{align}
\end{subequations}
This problem is (i.e., can be formulated as) a convex optimization problem  \cite{Ghosh_2006,Ghosh_2006b,JMLR:v15:osting14a}. 
Our goal is to investigate whether the $w^\star$-spectral graph realization, where the weight, $w^\star$ is the solution to  \eqref{e:EOP},  might have special properties.

\subsection*{Results}
We first show that the optimization problem in \eqref{e:MaxGraphRealization2} is well-posed.
\begin{thm} \label{t:well-posed}
The optimization problem in \eqref{e:MaxGraphRealization2} is well-posed; there exists an admissible $X^\star$ that attains the supremum value of the objective over the constraint set. Furthermore, if $w \in \mathbb R^m_+$ satisfies $w^t \phi = 1$, 
\begin{equation} \label{e:objFunBnd}
\|X^\star\|_F^2  \leq \frac{1 }{\lambda_2 (\Delta_w) }. 
\end{equation}
\end{thm}
A proof of  \cref{t:well-posed} will be given in \cref{s:Proofs}. 
We will see that the inequality in \eqref{e:objFunBnd} is a weak duality result for \eqref{e:MaxGraphRealization2} and \eqref{e:EOP}. 
Our main result is that the solution to \eqref{e:EOP} can be used to generate a maximal graph realization. 

\begin{thm} \label{t:UnitDistGraphReal}
Suppose $w^\star \geq 0$ solves \eqref{e:EOP} and the maximum eigenvalue $\lambda^\star$ has a $d$-dimensional eigenspace $E_{\lambda^\star}$. Then there exist orthogonal eigenvectors $x_\ell \in E_{\lambda^\star}$, $\ell \in [d]$ such that the $d$-dimensional graph realization $x\colon V \to \mathbb  R^d$ with coordinate matrix $X = (x_1 | \cdots | x_d) \in \mathbb R^{n \times d} $  
is a maximal graph realization solving \eqref{e:MaxGraphRealization2}. 
In particular, the graph realization  is centered, \ie, 
$\sum_{i \in [n]} x_\ell(i) = 0, \ \forall \ell \in [d]$ 
and satisfies the edge-length constraints, \ie, 
$\|x(i) - x(j) \|^2 = \phi_k, \ \forall \ k = (i,j) \in E$. 
Moreover, if  $w^\star > 0$, then $X \in \mathbb R^{n \times d}$ is a solution to  \eqref{e:MaxGraphRealization}.
\end{thm}

A proof of  \cref{t:UnitDistGraphReal} will be given in \cref{s:Proofs} and depends on the duality between \eqref{e:MaxGraphRealization2} and  \eqref{e:EOP}. 
The graphs in \cref{f:graphZoo} were generated using \cref{t:UnitDistGraphReal}; see \cref{s:NumExamples}.  
However, there are a couple of practical limitations of \cref{t:UnitDistGraphReal} relevant to its use the application of generating a unit-distance graph realization. For a given graph, before solving \eqref{e:EOP}, we do not know 
(i) the dimension $d = \dim E_{\lambda^\star}$ of the spectral graph realization or
(ii) if the assumption that  $w^\star > 0$ holds.
In \cref{s:NumExamples}, we will give examples of when the dimension of the graph realization is greater than 3 and when the assumption $w^\star > 0$ fails.  

A useful consequence of the strong duality used in the proof of  \cref{t:UnitDistGraphReal} is the following corollary.
\begin{cor} \label{c:Duality}
Suppose $X \in \mathbb R^{n \times d}$ and $w \in \mathbb R^m$ satisfy the following conditions: 
\begin{subequations}
\begin{align}
& X^t 1 = 0, \quad  \| X^t b_k \|^2 \leq \phi_k,  \ \forall k \in E  \\ 
& w \geq 0, \quad w^t \phi = 1, \quad d = \dim E_{\lambda_2(\Delta_w)} \\
\label{e:Dualityc}
&  \| X \|_F^2  = \frac{1}{\lambda_2(\Delta_w)}. 
\end{align}
\end{subequations}
Then $X$ and $w$ are respective solutions of \eqref{e:MaxGraphRealization2} and \eqref{e:EOP}. 
\end{cor}

Note that \eqref{e:Dualityc} is the saturation of the inequality \eqref{e:objFunBnd}. 
The following example shows how \cref{c:Duality} can be used to establish a  maximal graph realization for a cycle graph.  
\begin{exam} \label{ex:cycle}
A centered, regular $n$-gon with unit edge lengths satisfies 
$$
\|X\|_F^2 = \sum_{i \in [n]} \left( \frac{1}{2} \csc \left( \frac{\pi}{n}\right) \right)^2 = \frac{n}{4} \csc^2 \left( \frac{\pi}{n} \right). 
$$
The skeleton of the $n$-gon, an $n$-cycle, with edge weights $\frac{1}{n}$ satisfies 
$$
\lambda_2(\Delta_w ) =\frac{1}{n}  \lambda_2(\Delta_1) = \frac{4}{n}  \sin^2\left( \frac{\pi}{n} \right). 
$$
The corresponding eigenspace is $d=2$ dimensional.  
Since $ \| X \|_F^2  = \frac{1}{\lambda_2(\Delta_w)}$, by \cref{c:Duality}, we have that the maximal graph realization of the $n$-cycle is a centered, regular $n$-gon. 
\end{exam}

A partial converse to \cref{t:UnitDistGraphReal} can be established. 
We show that, under certain assumptions,  the coordinate vectors of a maximal graph realization are eigenvectors of a $w$-weighted graph Laplacian corresponding to the same eigenvalue for \emph{some} choice of graph weights $w \in \mathbb R^m_+$.
The main hurdle is that the dimension $d$ for \eqref{e:MaxGraphRealization2} is unknown. 
 To state our result, we require the following definition. 
\begin{defin} 
\label{d:regular}
For a realization of graph $G$, we say the coordinate  matrix $X \in \mathbb R^{n\times d}$ is \emph{regular} if there does not exist a weight $w\in \mathbb R^m \setminus \{0\}$ such that $\Delta_w X = 0$. 
\end{defin}
Note that in the definition of a regular coordinate matrix, we do not restrict the weights to be nonnegative. We also note that the equation $\Delta_w X = 0$ specifies $nd$ linear conditions on $m$ variables, so a regular coordinate matrix for graph with $n$ vertices and $m$ edges s in $\mathbb R^d$ must satisfy $m \leq nd$. 
\begin{thm} \label{t:Converse}
For any regular solution $X^\star \in \mathbb R^{n \times d}$ of  \eqref{e:MaxGraphRealization2}, there exists $w \geq 0$ such that 
$$
\Delta_w X^\star = X^\star,
$$
\ie,  the columns of $X^\star$ are eigenvectors of the $w$-weighted graph Laplacian corresponding to eigenvalue 1. Moreover, if for some $k \in E$, we have that $\| (X^\star)^t b_k \|^2 < \phi_k$, then $w_k = 0$. 
\end{thm}
A proof of  \cref{t:Converse} will be given in \cref{s:Proofs}. 

\bigskip

We can also consider the problem of minimizing $\lambda_{max}(\Delta_w)$ over the edge weights, $w$.
\begin{thm} \label{t:UnitDistGraphReal'}
Suppose $w^\star \geq 0$ solves 
\begin{subequations}
\label{e:EOP'}
\begin{align}
\min_{w \in \mathbb R^m} \ &  \lambda_n( \Delta_w) \\ 
\textrm{s.t.} \ & \Delta_w u_i = \lambda_i   u _i, \qquad \forall i \in [n]\\ 
& w \geq 0, \quad w^t \phi = 1. 
\end{align}
\end{subequations}
and the minimum eigenvalue $\lambda^\star$ has a $d$-dimensional eigenspace $E_{\lambda^\star}$. Then there exist orthogonal eigenvectors $x_\ell \in E_{\lambda^\star}$, $\ell \in [d]$ such that the $d$-dimensional graph realization $x\colon V \to \mathbb  R^d$ with coordinate matrix $X = [x_1 \mid \cdots \mid x_d] \in \mathbb R^{n \times d} $  
is a minimal graph realization solving 
\begin{subequations}
\label{e:MaxGraphRealization2'}
\begin{align}
\min_{X \in \mathbb R^{n \times d}} \ & \| X \|_F^2 \\
\label{e:2b'}
\textrm{s.t.} \ & \| X^t b_k \|^2 \geq \phi_k,  \qquad \forall k \in E \\ 
& X^t 1 = 0.
\end{align}
\end{subequations}
Moreover, if  $w^\star > 0$, then $X \in \mathbb R^{n \times d}$ satisfies 
$\| X^t b_k \|^2 = \phi_k$, $\forall k \in E$.
\end{thm}
The proof of  \cref{t:UnitDistGraphReal'} is similar to the proof of \cref{t:UnitDistGraphReal} and will be omitted. 
In \cref{s:NumExamples}, we give several examples to illustrate  \cref{t:UnitDistGraphReal'}. 
Analogous to \cref{c:Duality}, the following corollary is a consequence of the strong duality used in the proof of   \cref{t:UnitDistGraphReal'}. 
\begin{cor} \label{c:Duality'}
Suppose $X \in \mathbb R^{n \times d}$ and $w \in \mathbb R^m$ satisfy the following conditions: 
\begin{subequations}
\begin{align}
& X^t 1 = 0, \quad  \| X^t b_k \|^2 \geq \phi_k,  \ \forall k \in E  \\ 
& w \geq 0, \quad w^t \phi = 1, \quad d = \dim E_{\lambda_2(\Delta_w)} \\
&  \| X \|_F^2  = \frac{1}{\lambda_n(\Delta_w)}. 
\end{align}
\end{subequations}
Then $X$ and $w$ are respective solutions of \eqref{e:MaxGraphRealization2'} and \eqref{e:EOP'}. 
\end{cor}

The following example shows how \cref{c:Duality'} can be used to establish a  minimal graph realization for a semi-regular bipartite graph. 
\begin{exam} \label{ex:2pt}
Let $G$ be a $(d_+,d_-)$-semi-regular bipartite graph with  bipartition $V = V_+ \cup V_-$. 
Recall that a bipartite graph with bipartition $V = V_+ \cup V_-$ is $(d_+,d_-)$-\emph{semi-regular} if all vertices in $V_+$ have degree $d_+$ and all vertices in $V_-$ have degree $d_-$.  
We claim the the minimal graph realization solving \eqref{e:MaxGraphRealization2'} with $\phi = 1$ is given by 
\begin{equation}
\label{e:MinGraphRealX}
x(i) = 
\begin{cases} 
c +\frac{1}{2}, & i \in V_+ \\ 
c -\frac{1}{2}, & i \in V_-
\end{cases}, 
\end{equation}
where $c = - \frac{|V_+| -|V_-|}{2 ( |V_+| + |V_-|)}$. 
First observe that the graph realization in \eqref{e:MinGraphRealX} is centered:
$$
\sum_{i \in [n]} x(i) = |V_+| \left(c+\frac{1}{2} \right) + |V_-| \left( c - \frac{1}{2} \right) = 0. 
$$
 We then compute
 $$
 \|X\|_F^2 = \sum_{i \in [n]} x(i)^2 = \frac{|V_+| |V_-| }{ |V_+| + |V_-| }.
$$ 

For a $(d_+,d_-)$ semi-regular graph and 
 $w_k = \frac{1}{m}$, $\forall k \in [m]$,  we have
 $\lambda_n(\Delta_w) =  \frac{1}{m} \lambda_n (\Delta_1)   = \frac{1}{m} \left( d_+ + d_- \right)$; 
 see \cite{doi:10.1080/03081088508817681} or \cite[Thm. 2.2(b)]{Mohar91thelaplacian}. Using that $d_+|V_+| = d_- |V_-| = m$, we obtain 
 $$  
 \lambda_n(\Delta_w) =  \frac{|V_+| + |V_-|}{ |V_+| |V_-|} = \frac{1}{ \| X \|_F^2 }.
 $$ 
 The corresponding eigenspace is $d=1$ dimensional. 
By  \cref{c:Duality'}, the two point realization in  \eqref{e:MinGraphRealX} is a minimal graph realization. 
\end{exam}

\subsection{Related work} 
The second eigenvalue of the graph Laplacian is referred to as the \emph{algebraic connectivity} and is related to the other notions of graph connectivity \cite{Fiedler_1973}. The problem \eqref{e:EOP} of maximizing the algebraic connectivity as a function of the edge weights  has been considered in a variety of previous papers; see, for example,  \cite{Ghosh_2006,Ghosh_2006b,JMLR:v15:osting14a}. However, the optimality conditions have not previously been interpreted in terms of edge-length constrained spectral realizations.   In \cite{Osting_2017}, the authors consider maximizing spectral quantities of a weighted graph Laplacian, where the edge weights depend on the distance between adjacent vertices in a graph realization. Here, we take the edge weights to vary independently of the coordinates of the graph realization. 

There are deep connections between the symmetry properties of a graph (automorphism group) and the eigendecomposition of the graph Laplacian \cite{Terwilliger_1982,Van_Dam_2016}. 

Our primary motivation is from recent results in spectral geometry.  In particular, A. Frasier and R. Schoen recently showed that the solution of an extremal Steklov eigenvalue problem on a closed surface with boundary can be used to generate a free boundary minimal surface (FBMS) \cite{Fraser_2015}. 
Using this connection, the author together with \'E. Oudet and C.Y. Kao computed many FBMS \cite{Oudet_2021}. 
 Analogously, R. Petrides showed that the solution of an extremal Laplace-Beltrami eigenvalue problem generates a minimal isometric immersion into some $d$-sphere by first eigenfunctions \cite{Petrides_2014}. The problem considered in this paper can be considered as a graph version of these results.

\subsection*{Outline}
In \cref{s:Proofs}, we prove \cref{t:well-posed}, \cref{t:UnitDistGraphReal}, and  \cref{t:Converse}. 
We give several numerical examples in \cref{s:NumExamples}. 
We conclude in \cref{s:Disc} with a brief discussion.

\section{Proofs of Theorems~\ref{t:well-posed}, \ref{t:UnitDistGraphReal}, and \ref{t:Converse}.}
\label{s:Proofs}

\begin{proof}[Proof of Theorem~\ref{t:well-posed}.]
For $X \in \mathbb R^{n \times d}$ satisfying the constraints in \eqref{e:MaxGraphRealization2}, and $w \in \mathbb R^m_+$ satisfying $w^t \phi = 1$, we compute
$$
1 = \sum_{k \in [m]} w_k \phi_k =  \sum_{k\in[m]} w_k b_k^t X X^t b_k = \textrm{tr}\left( \left( \sum_{k \in [m]} w_k b_k b_k^t \right) X X^t \right) = \textrm{tr} \left( X^t \Delta_w X\right) \geq \lambda_2(\Delta_w) \| X\|_F^2 
$$
which gives \eqref{e:objFunBnd}. The result then follows from the Bolzano–Weierstrass theorem.
\end{proof}

\begin{proof}[Proof of Theorem~\ref{t:UnitDistGraphReal}.]
It is equivalent to write  \eqref{e:EOP} as the convex semidefinite program (SDP)
\begin{subequations}
\label{e:SDP}
\begin{align}
\max_{t,w}  \  & t \\
s.t \ & \Delta_w \succeq t J  \\
& w \geq 0, \quad \phi^t w = 1. 
\end{align}
\end{subequations}
Here, $J  \in \mathbb S^n_+$ is the rank $n-1$ project matrix given by $J= I - 11^t/n$. 

Our proof relies on the Lagrange dual formulation of \eqref{e:SDP}. 
For the dual variables $Y \in \mathbb S^n_+ $ and $\mu \in \mathbb R$, we consider the \emph{Lagrangian}
\begin{align*}
L(w,t; Y, \mu) &:= t - \langle t J - \Delta_w, Y \rangle_F - \mu ( \phi^t w  - 1) \\
& = \mu + t \left( 1 - \langle J, Y \rangle_F \right) + \sum_{k \in [m]} w_k ( b_k^t Y b_k - \mu \phi_k) .
\end{align*}
For fixed $Y \in \mathbb S^n_+ $ and $\mu \in \mathbb R$, the \emph{dual function} is defined 
\begin{align*}
g(Y,\mu) &:= \max_{w\geq 0, t} \ L(w, t; Y, \mu) \\
&= \begin{cases}
\mu, &  \langle J , Y \rangle_F  = 1, \ b_k^t Y b_k \leq \mu \phi_k \\ 
\infty, & \textrm{otherwise}
\end{cases}.
\end{align*}
The \emph{dual problem} can then be written 
\begin{subequations}
\label{e:DualProb}
\begin{align}
\min_{Y\in \mathbb S^n_+, \mu } \ & \mu \\ 
\textrm{s.t.} \ & \langle J, Y \rangle_F =1 \\
& b_k ^t Y b_k \leq \mu \phi_k, \quad \forall k \in [m]. 
\end{align}
Note that replacing $Y$ by $Y + \alpha 1 1^t$ for any $\alpha > 0$ in \eqref{e:DualProb}  does not change the objective value and still satisfies all constraints. Thus, we may augment \eqref{e:DualProb} with the additional constraint, 
\begin{equation}
1^t Y 1 = \langle 1 1^t,  Y \rangle_F = 0
\end{equation}
\end{subequations}

The semi-definite optimization problem in \eqref{e:SDP} is a convex. Furthermore, the following argument shows that Slater's condition (see, \eg, \cite[Section 5.2.3]{Boyd_2004}) is satisfied. 
Define $\hat w \in \mathbb R^m_+$ by 
$\hat w_k = \frac{1}{\phi^t 1}$, $k\in [m]$ 
so that  $\hat w^t \phi = 1$. 
Define $\hat t = \frac{1}{2} \frac{1}{\phi^t 1} \lambda_2(\Delta_1)$ and note that the connectivity assumption on $G$ guarantees that $\hat t > 0$.  
We then compute 
$$
\Delta_{\hat w} =  \frac{1}{\phi^t 1} \Delta_1 \succeq \frac{1}{\phi^t 1} \lambda_2(\Delta_1) J \succ \hat tJ,
$$
which shows that $\hat w, \hat t$ satisfy Slater's condition for the constraints in \eqref{e:SDP}. 

The primal and dual convex optimization problems  (\eqref{e:SDP} and \eqref{e:DualProb}) then satisfy strong duality, \ie, denoting their optimal values  with the superscript $\star$, we have that $t^\star = \mu^\star$. 
It is also necessary and sufficient that the Karush-Kuhn-Tucker (KKT) conditions be satisfied by the optimal solution. The KKT conditions can be stated:  there exist 
primal variables 
$w^\star \in  \mathbb R^m$, 
$t^\star \in \mathbb R$
and dual variables 
$Y^\star \in \mathbb S^n_+$, 
$\mu^\star \in \mathbb R$, 
satisfying 
\begin{subequations}
\label{e:KKT}
\begin{align}
\label{e:KKTa}
& \Delta_{w^\star} \geq t^\star J,  \ w^\star \geq 0,  \ \phi^t w^\star = 1 \\
\label{e:KKTb}
& Y^\star \in \mathbb S^n_+, \ \langle J, Y^\star \rangle_F = 1, \ 1^t Y 1 = 0,  \ b_k ^t Y^\star b_k \leq \mu^\star \phi_k, \ \forall k \in [m] \\ 
\label{e:KKTc}
& w_k \left( b_k ^t Y b_k - \mu \phi_k \right) = 0, \ \forall k \in [m] . 
\end{align}
\end{subequations}
Here, we've grouped the conditions as primal feasible \eqref{e:KKTa}, dual feasible \eqref{e:KKTb}, and complimentary slackness \eqref{e:KKTc}. 

For a solution to the KKT conditions, we have 
\begin{align}
t^\star 
= t^\star \langle J, Y^\star \rangle_F 
\leq \langle  \Delta_{w^\star}, Y^\star \rangle_F
= \sum_{k\in [m]} w^\star_k b_k^t Y^\star b_k = \mu^\star \sum_k w^\star_k \phi_k = \mu^\star. 
\end{align}
By strong duality, $t^\star = \mu^\star  = \lambda_2^\star$, which implies that 
\begin{equation}
\langle   \Delta_{w^\star}, Y^\star \rangle_F = \lambda_2^\star.
\end{equation}
Noting that $1$ is an eigenvector of $Y^\star$ with eigenvalue 0, 
we can write the eigenvalue decomposition $Y^\star = \sum_{i \in [n-1]} \beta_i v_i \otimes v_i$ with $\|v_i\| = 1$ and  $v_i \perp 1$. 
We have
$$
\lambda_2^\star 
=  \langle \Delta_{w^\star}, Y^\star  \rangle_F 
=  \sum_{i \in [n-1]} \beta_i  \langle v_i, \Delta_{w^\star} v_i \rangle 
\geq  \sum_{i \in [n-1]} \beta_i \lambda_2^\star \| v_i \|^2 = \lambda_2^\star,
$$
which implies that either $\beta_i = 0 $ or $v_i \in E_{\lambda^\star}$. Let $d = \dim  E_{\lambda_2^\star}$.  Relabelling if necessary, we have that 
$$
Y^\star = \sum_{i \in [d]}  \beta_i v_i \otimes v_i,  
\qquad \beta_i \geq 0, \  v_i \in E_{\lambda^\star}, \ \forall i \in [d].
$$
This implies that $Y^\star$ has rank at most $d<n$. 
Making the substitution 
$$
Y = \mu X X^t, \qquad X \in \mathbb R^{n \times d}
$$ 
in \eqref{e:DualProb},  it follows that \eqref{e:DualProb} is equivalent to the optimization problem 
\begin{subequations}
\label{e:D3}
\begin{align}
\min_{\mu \in \mathbb R, X \in \mathbb R^{n \times d}}  \ & \mu \\ 
\textrm{s.t.} \ & \mu\| X\|_F^2 = 1 \\
& \| X^t b_k \|^2 \leq \phi_k, \quad \forall k \in [m] \\ 
& X^t 1 = 0.
\end{align}
\end{subequations}
Here,  
$$
X = \frac{1}{\sqrt{\mu^\star} } \left[ \sqrt{\beta_1} v_1 | \cdots |   \sqrt{\beta_d} v_d \right]  \in \mathbb R^{n \times d}
$$ 
is a  coordinate matrix for a centered realization of the graph.
Eliminating $\mu$ by setting $\mu = \frac{1}{\|X\|_F}$, this can be equivalently rewritten as \eqref{e:MaxGraphRealization2}. 
 If $w^\star > 0$, then by the 
complimentary slackness condition \eqref{e:KKTc}, it is a  unit-distance realization. 
In this case, the solution to \eqref{e:MaxGraphRealization2} is also a solution to \eqref{e:MaxGraphRealization}, since the latter constraint set is smaller. 
\end{proof}

 \begin{proof}[Proof of Theorem~\ref{t:Converse}.]
We first observe that the condition that the realization be centered is necessary. For any $w>0$, taking the inner product of both sides of $\Delta_w u_\ell = u_\ell$ with the ones vector
 gives that 
$$
\langle 1, u_\ell \rangle 
= \langle 1, \Delta_w u_\ell \rangle 
= \langle \Delta_w 1,  u_\ell \rangle 
= 0, 
\qquad \forall \ell \in [d]. 
$$

We will consider the Karush-Kuhn-Tucker (KKT)  conditions for \eqref{e:MaxGraphRealization2}, which are necessarily satisfied for a stationary point satisfying constraint qualifications. 
We use the constraint qualification that the gradients with respect to the active inequality constraints and equality constraints are linearly independent. Writing $\|X^t b_k \|^2 = \textrm{tr} (b_k b_k^t X X^t ) $, we compute
\begin{equation}
\label{e:grad1}
\nabla_X \|X^t b_k \|^2 = 2 b_k b_k^t X, \qquad \forall k \in [m]. 
\end{equation}
 Writing $X = [x_1 \mid \cdots \mid x_d ] \in \mathbb R^{n \times d}$ and  $X^t 1_n = \left[ x_1^t 1_n \mid \cdots \mid x_d^t 1_n \right]^t$, we compute
\begin{equation}
\label{e:grad2}
\nabla_X x_\ell^t 1_n = 1_n e_\ell^t, \qquad \forall \ell \in [d].
\end{equation}
Note that the gradients in \eqref{e:grad1} and \eqref{e:grad2} are linearly independent since $b_k^t 1_n = 0$. The condition that the gradients in \eqref{e:grad1} are linearly independent can be stated: there does not exist a $w \in \mathbb R^m \setminus \{0\}$ such that \
$$
\sum_{k \in [m]} w_k b_k b_k^t X = 0_{n\times d}.  
$$ 
Since $\Delta_w = \sum_{k \in [m]} w_k b_k b_k^t$, this is precisely the regularity of $X$ (\cref{d:regular}). 

For dual variables $z \in \mathbb R^d$ and $w \in \mathbb R^m_+$, we introduce the \emph{Lagrangian}, 
\begin{subequations}
\begin{align}
L(X; z, w) & := \| X \|_F^2 + z^t X^t 1 - \sum_{k \in {[m]}} w_k \left( \| X^t b_k \|^2 - \phi_k \right) \\
& = \textrm{tr}( X^tX) +  \textrm{tr}(1 z^t X^t) - \left(  \textrm{tr}(\Delta_w X X^t) - w^t \phi \right).
\end{align}
\end{subequations}
The KKT conditions are that there exist $z \in \mathbb R^d$ and $w \in \mathbb R^m_+$ such that 
\begin{subequations}
\begin{align}
\label{e:Ka}
& \nabla_X L(X;z,w) = 2 X + 1 z^t - 2 \Delta_w X = 0\\
\label{e:Kb}
& w_k \left( \| X^t b_k \|^2 - \phi_k \right) = 0. 
\end{align}
\end{subequations}
Multiplying both sides of \eqref{e:Ka} on the left by $1 \in \mathbb R^n$ gives $z = 0$, yielding $\Delta_w X = X$. The complementary slackness condition \eqref{e:Kb} gives the final statement of the theorem. 
\end{proof}

\section{Numerical Examples} 
\label{s:NumExamples}
Since \eqref{e:SDP} is a convex optimization problem, it can easily be solved using \verb+CVXPY+ \cite{diamond2016cvxpy}. 
To represent the graphs, we use the \verb+networkx+ library \cite{networkx}. 
Our implementation is available at the author's github page \cite{Osting-github}.
We report here the result of several small-scale numerical experiments to illustrate our ideas. 
All experiments were performed in a small number of seconds on a laptop with a 1.6GHz Intel Core i5 processor and 8GB of RAM. 

\subsection{Some two- and three-dimensional graph realizations}
\label{s:NumExamples:A}
By \cref{ex:cycle}, a maximal graph realization for a $n$-cycle graph is the centered regular $n$-gon. 
We computationally verified this for several values of $n$.   

We considered several other graphs, including the 
a square lattice graph, 
a triangular lattice graph, 
a hexagonal lattice graph, 
a tetrahedral graph, 
a cube graph, 
a dodecahedral graph, 
an icosahedral graph, 
an octahedral graph, 
a buckyball graph, 
a truncated tetrahedron graph, 
a truncated cube graph, and 
a circular ladder graph. 
For each of these graphs, we solved the eigenvalue optimization problem \eqref{e:EOP} with $\phi = 1$, which gave a solution with $w^\star > 0$. 
By \cref{t:UnitDistGraphReal}, there exist orthogonal eigenvectors so that the spectral graph realization is a solution to 
\eqref{e:MaxGraphRealization}. 
The resulting unit-distance graph realizations are plotted in \cref{f:graphZoo}. 
 The edge colors represent the edge weights, $w^\star$. 
 Interestingly, while the square and triangular lattices have the regular embedding, the hexagonal one does not. The Platonic polyhedron graphs all have edges with equal weight, as expected. In the buckyball embedding, the edges adjacent to only hexagons and edges adjacent to both hexagons and pentagons have different weights. For the truncated tetrahedron and cube graphs, the edges resulting from truncations have different weight. For a circular ladder graph, we again observe different weights on the two types of edges.

\subsection{Petersen graph} 
\begin{wrapfigure}{r}{0.2\textwidth}
\vspace{-.5cm}
\centering
\includegraphics[width=0.18\textwidth]{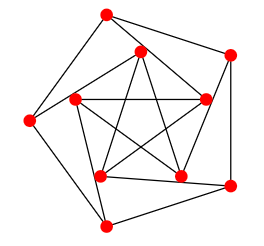}
\end{wrapfigure}
The \emph{Petersen graph} has  10 vertices and 15 edges and is known to have the two-dimensional unit-distance realization drawn on the right. 
The solution to  \eqref{e:EOP} with $\phi = 1$ yields a solution with $w^\star > 0$ and all equal optimal weights. 
But the multiplicity of the optimal eigenvalue is 5, giving a spectral realization in $d=5$ dimensions. 
The total variance for the five-dimensional maximal graph realization is  
$\| X \|_F = 7.5$ 
and the total variance for the two-dimensional realization is 
$\| X \|_F^2 = 5$.

\subsection{House graph and house x-graph} 
\label{s:HouseGraph}

Here we consider the 
house graph and the house x-graph. 

The \emph{house graph} has $n=5$ vertices and  $m=6$ edges and the usual two-dimensional realization is given in \cref{f:houseGraphs}(left). For the house graph, the solution to  \eqref{e:EOP} with $\phi = 1$ yields a solution with $w^\star > 0$. The maximal spectral realization is the usual drawing of the graph as in \cref{f:houseGraphs}(left). 

\begin{wrapfigure}{r}{0.2\textwidth}
\begin{center}
\includegraphics[width=0.11\textwidth]{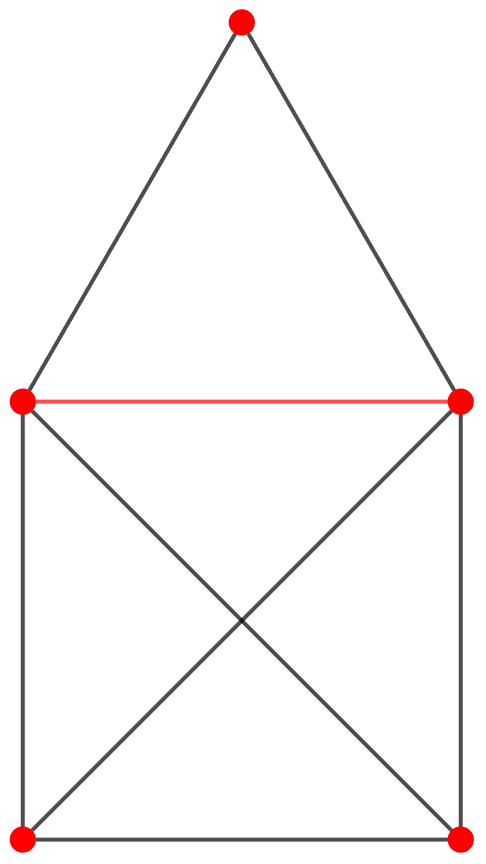}
\end{center}
\vspace{-.2cm}
\end{wrapfigure}
The \emph{house x-graph} has $n=5$ vertices and $m=8$ edges and the usual two-dimensional realization of the house-x graph is given on the right. It is the house graph with two additional edges incident diagonally opposite vertices of the square. 
For the house x-graph, the solution to  \eqref{e:EOP} with $\phi = 1$ yields a solution where 
one of the edge weights $w^\star_k$ is zero. The edge with zero weight is at the top of the square, drawn in red to the right. The maximal spectral realization is given in \cref{f:houseGraphs}(right). 
For the edge $k$ with zero weight, we have $\| X^t b_k\| = 0$, so that the two vertices at the top left and top right of the square are mapped to the same position.

\begin{figure}[t!]
\begin{center}
\includegraphics[width=0.35\textwidth]{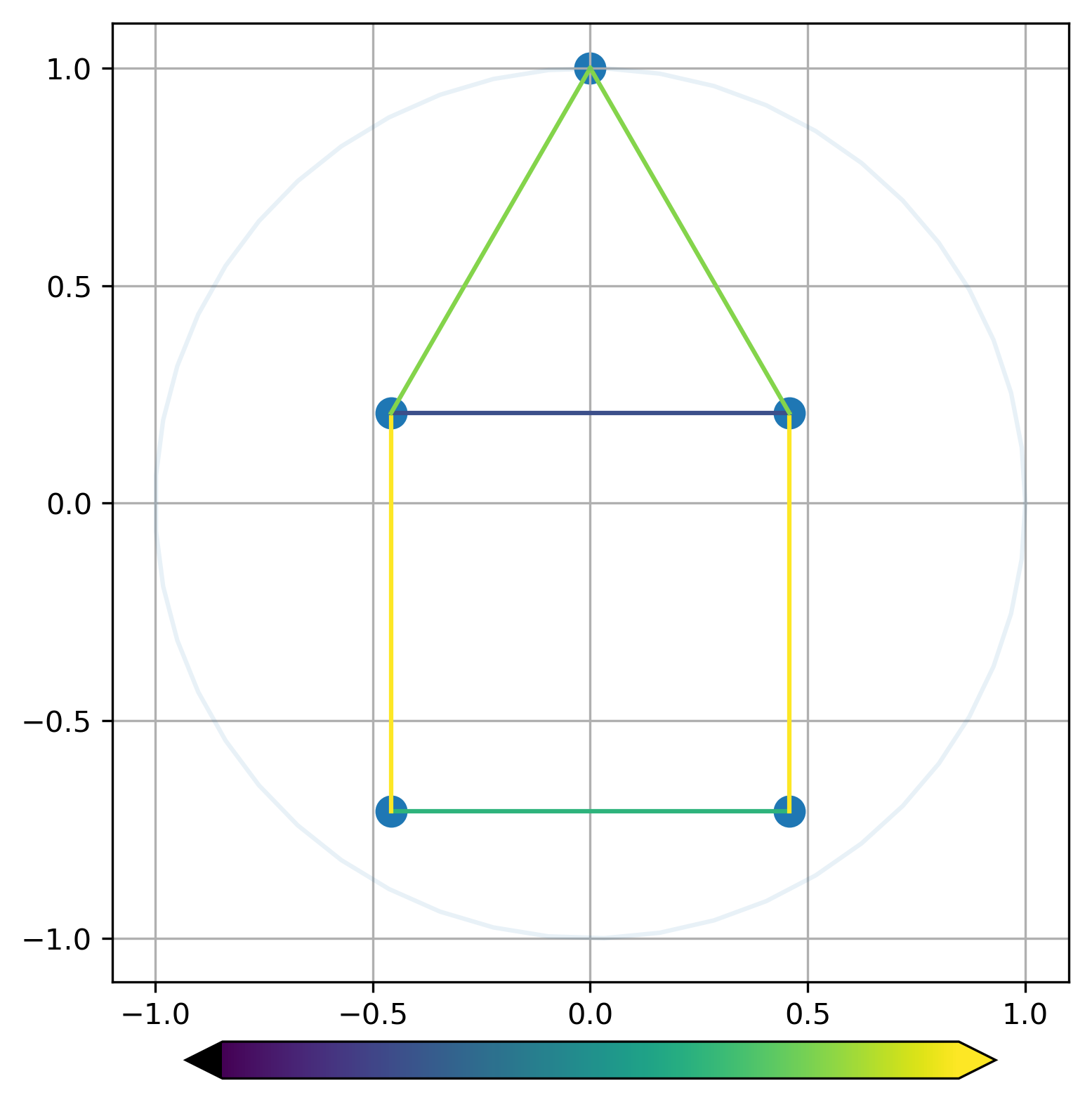}
\includegraphics[width=0.35\textwidth]{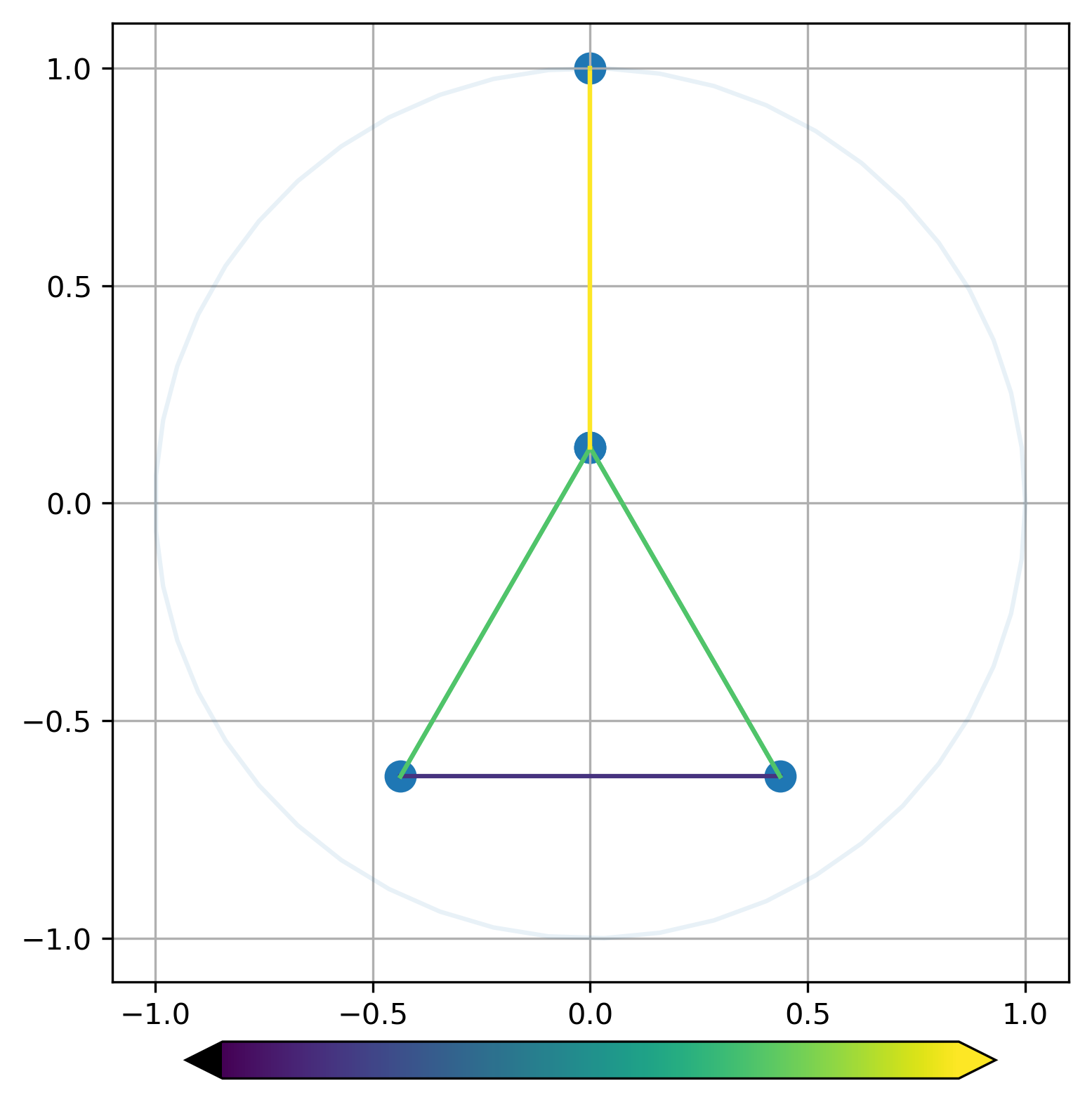}
\caption{Maximal spectral realizations of the house graph and the house x-graph. 
Both plots have been scaled so that the largest magnitude vertex has magnitude 1. 
See \cref{s:HouseGraph} for details.}
\label{f:houseGraphs}
\end{center}
\end{figure}

\subsection{Non-unit edge length constraints} 
\label{s:NonUnit}

\begin{figure}[t!]
\begin{center}
\includegraphics[width=0.3\textwidth]{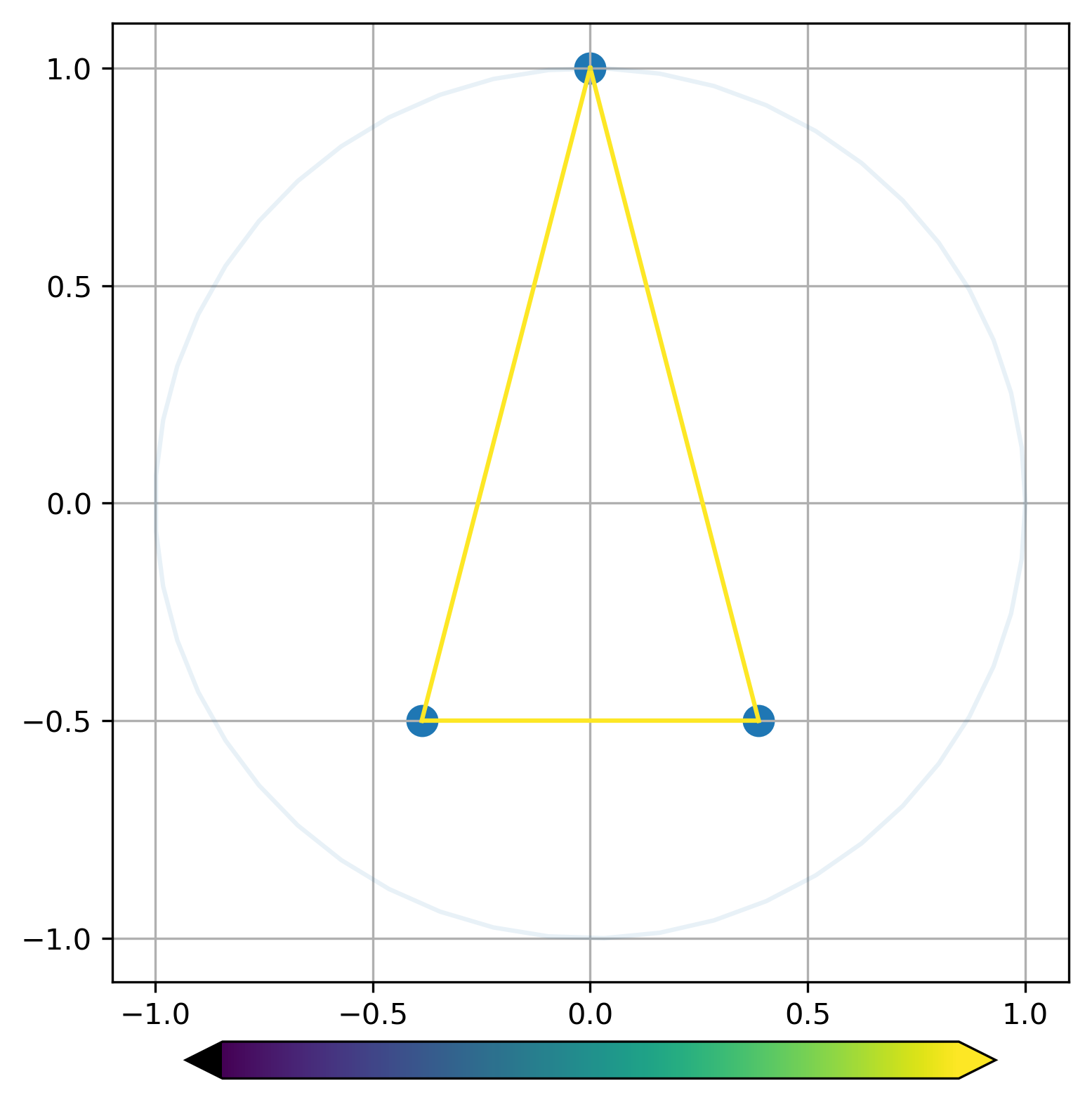}
\includegraphics[width=0.3\textwidth]{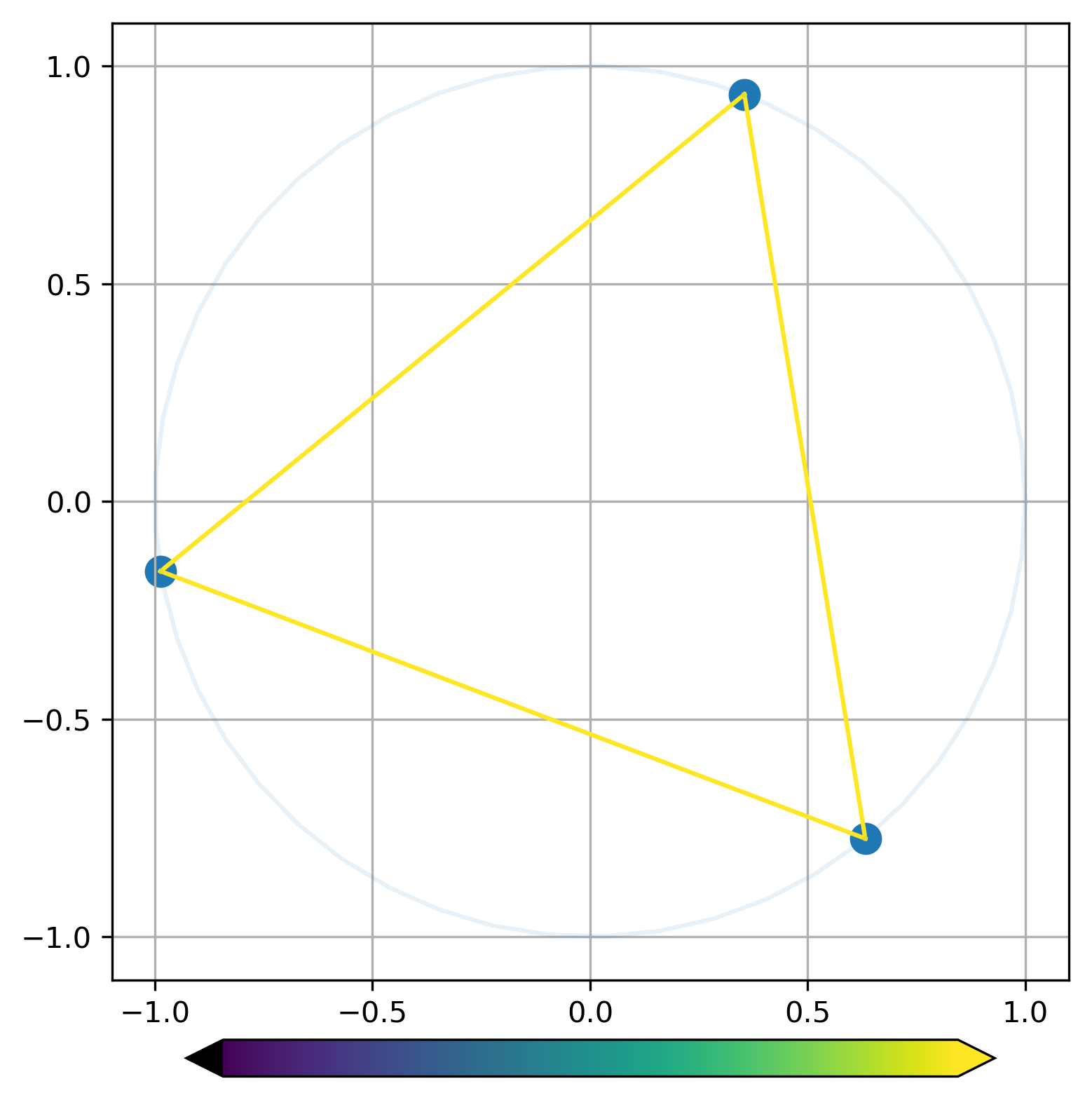}
\includegraphics[width=0.3\textwidth]{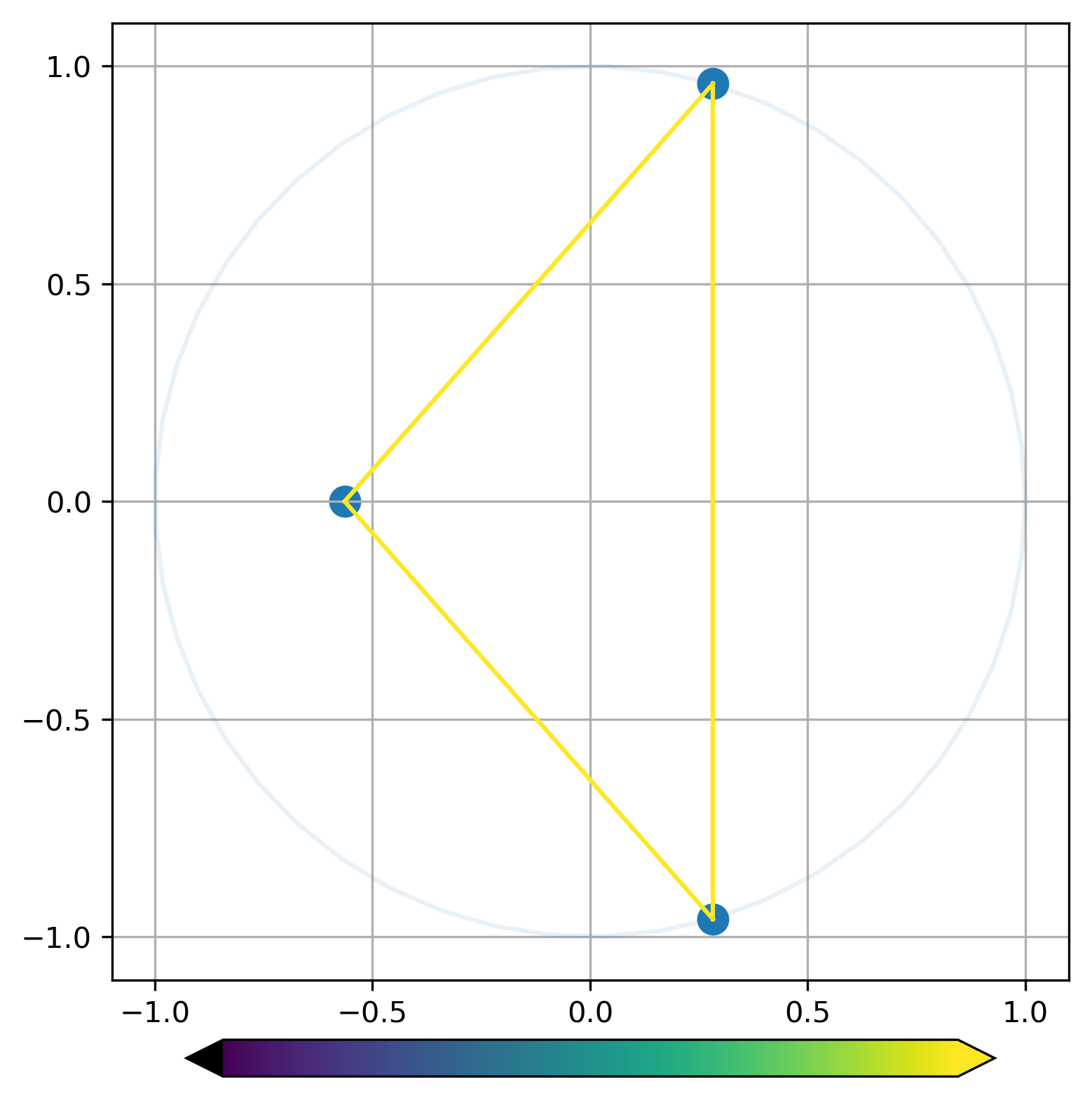}
\caption{Maximal spectral realizations of the 3-cycle with non-unit edge length constraints. 
All plots have been scaled so that the largest magnitude vertex has magnitude 1. 
See \cref{s:NonUnit} for details.}
\label{f:triangleGraphs}
\end{center}
\end{figure}

To illustrate the problem where the edge length constraints are not unity, \ie, $\phi \neq 1$,  we consider the cycle on $n=3$ vertices and 
$\phi = (a,1,1)$ for $a\in(0,2)$. 
The maximal graph realization are plotted in \cref{f:triangleGraphs} for $a=0.5$ (left), $a=1$ (center), and $a=1.5$ (right). 
For $a > 2$, a realization with these edge length constraints is not possible by the triangle inequality. In this case, one of the edge weights becomes zero and the realization becomes one-dimensional.

\subsection{Minimizing the largest graph Laplacian eigenvalue and minimal graph realizations} 
\label{s:NumExamples:Largest}

\begin{figure}[t!]
\begin{center}
\includegraphics[width=0.45\textwidth]{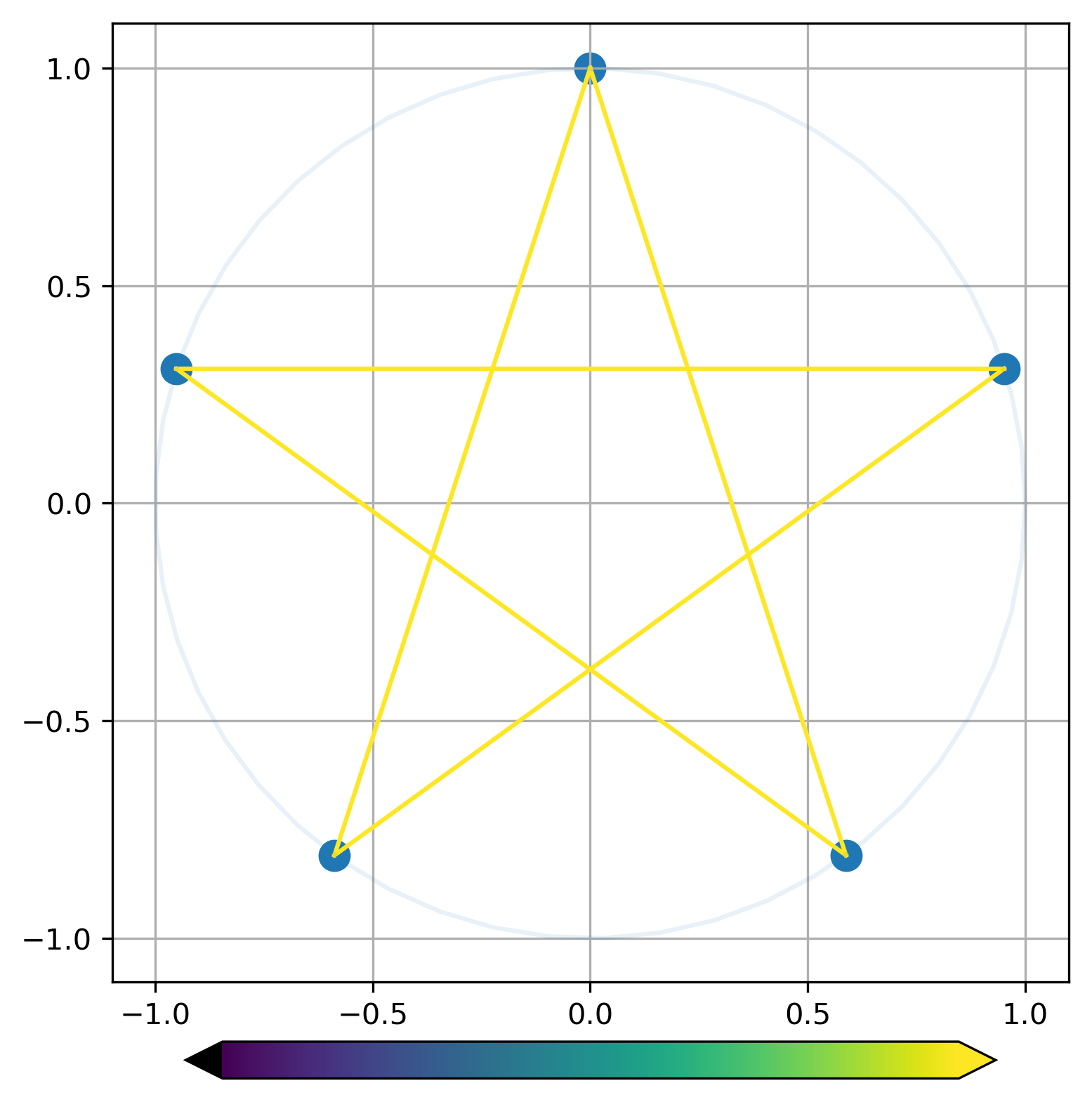}
\includegraphics[width=0.45\textwidth]{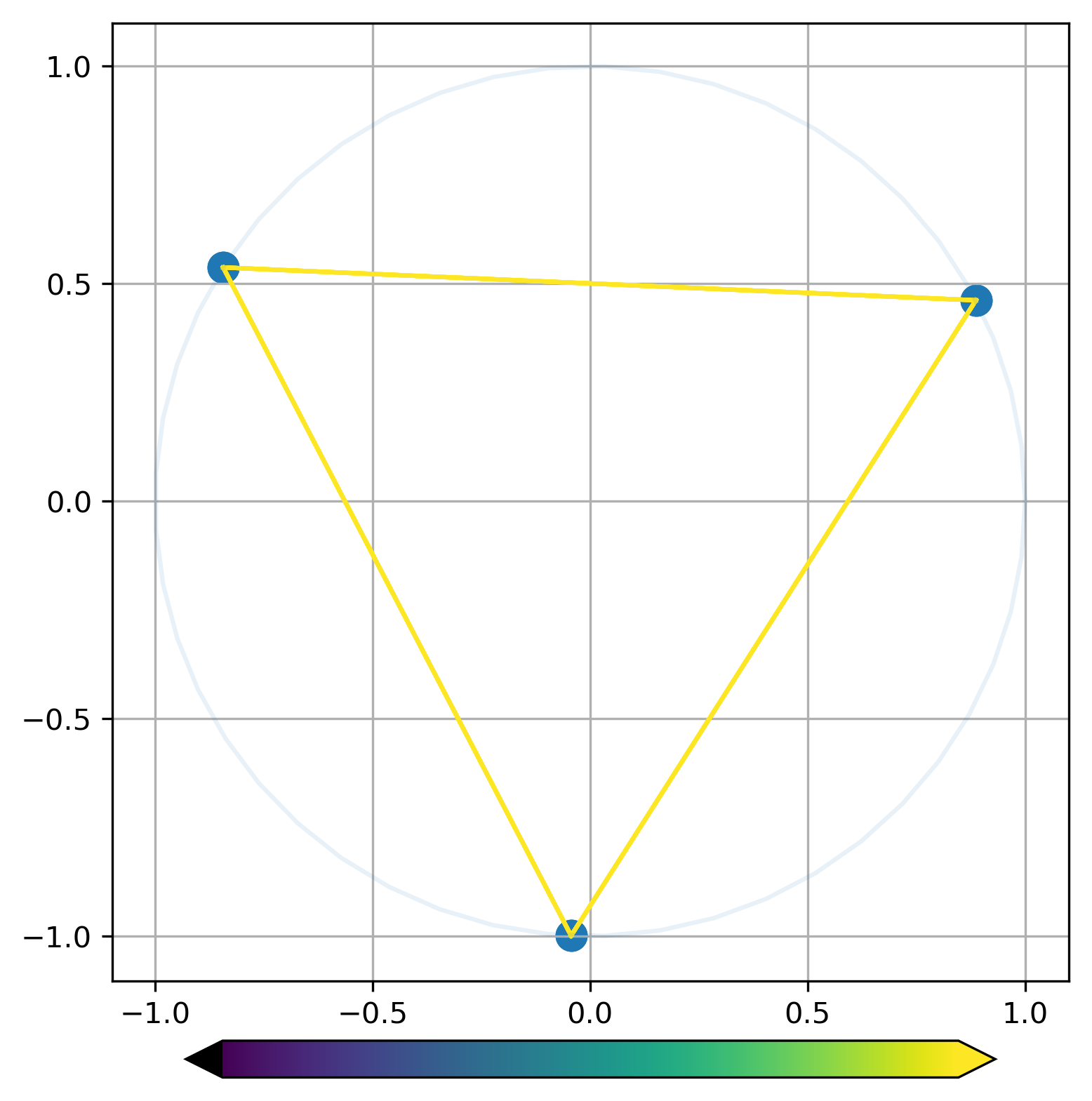} \\
\includegraphics[width=0.45\textwidth]{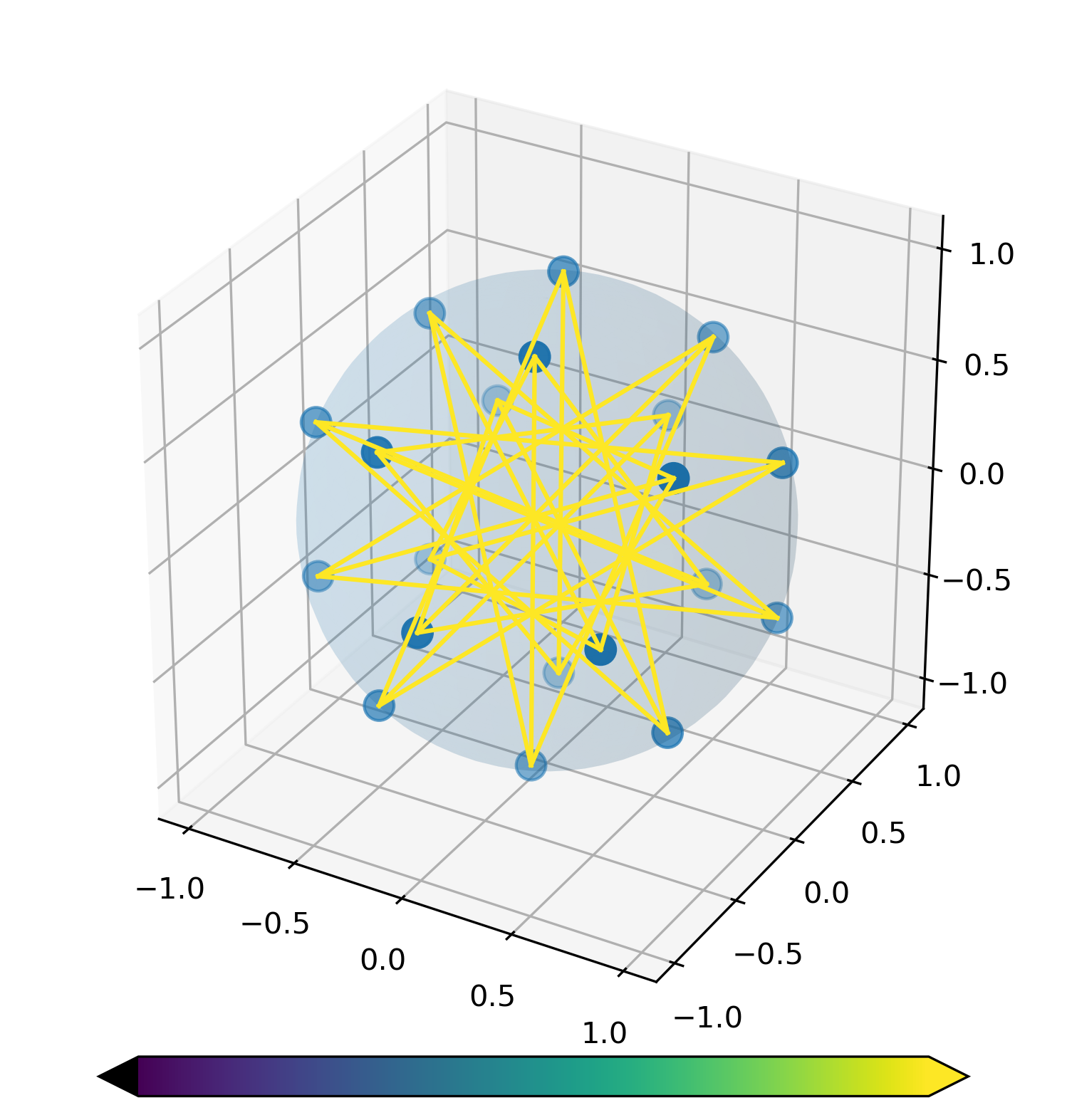}
\includegraphics[width=0.45\textwidth]{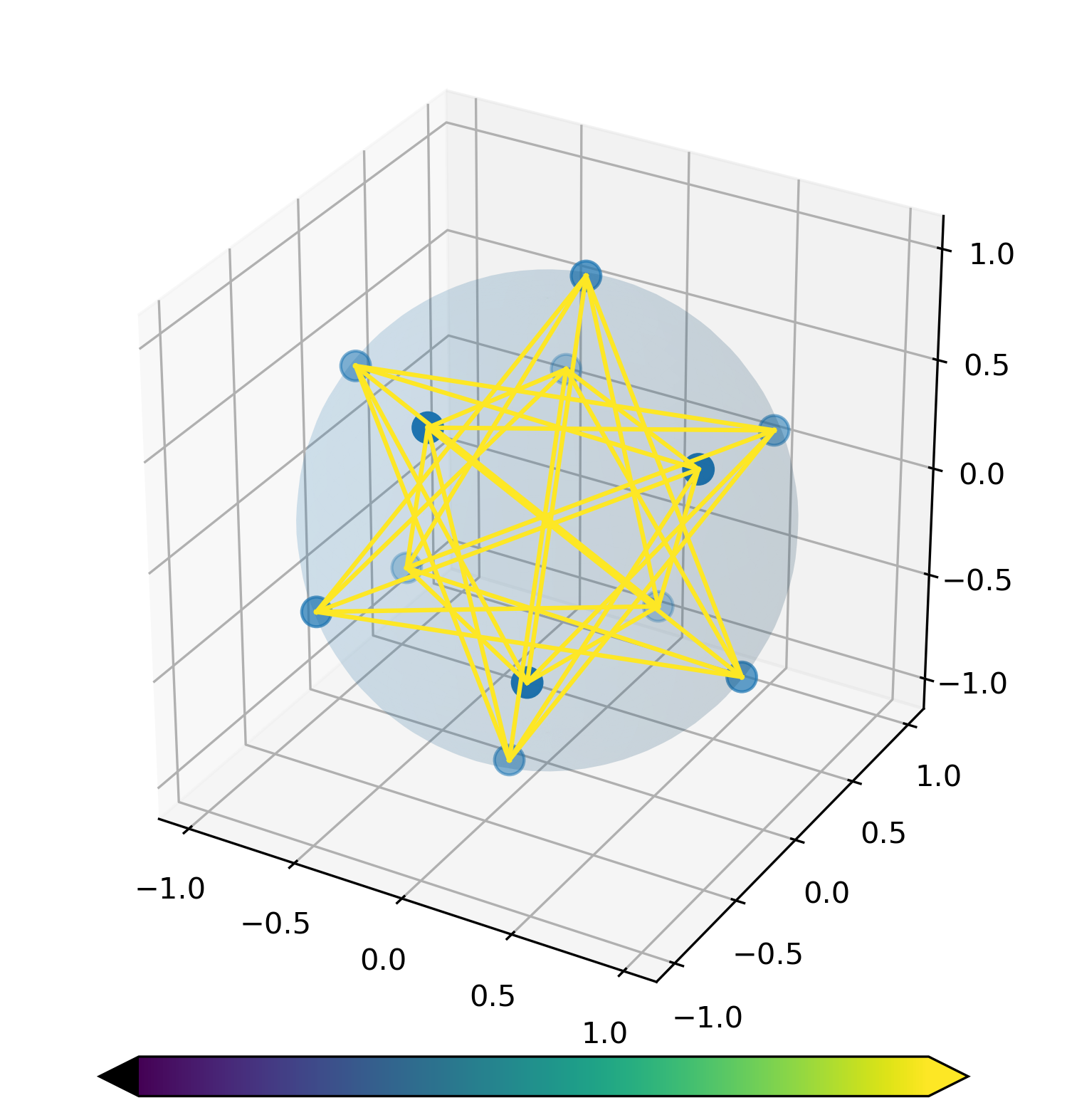} 
\caption{Minimal graph realizations for a variety of graphs: 
a 5-cycle graph, 
an octahedral graph, 
a dodecahedral graph, and
an icosahedral graph. 
All plots have been scaled so that the largest magnitude vertex has magnitude 1. 
See \cref{s:NumExamples:Largest} for more details.}
\label{f:graphZoo'}
\end{center}
\end{figure}

We consider the problems of minimizing the largest eigenvalue of the graph Laplacian \eqref{e:EOP'} and minimal graph realizations \eqref{e:MaxGraphRealization2'}. 
By \cref{ex:2pt}, a minimal graph realization for a semi-regular bipartite graphs is the two point realization. 
We computationally verified this for the cube graph, the square lattice graph, and the $n$-cycle graph for several even values of $n$.   

In \cref{f:graphZoo'}, we plot the minimal graph realizations for the 5-cycle graph, the octahedral graph, the dodecahedral graph, and the icosahedral graph. 
The minimal graph realizations 
generated using \cref{t:UnitDistGraphReal'}
are more compact than the maximal graph realization generated using 
\cref{t:UnitDistGraphReal} and hence generally have more crossing edges, are more ``spiky'', and are not injective.  
We also found the minimal graph realization for a tetrahedral graph, but it is the same as the maximal graph realization (the usual graph embedding). This might be anticipated because there is a unique (modulo rotation) embedding of the tetrahedral graph in three dimensions.

\section{Discussion} \label{s:Disc}
In this paper, for a fixed graph $G$, we considered the relationship between 
maximal graph realizations satisfying \eqref{e:MaxGraphRealization2}
and the  optimization problem of maximizing the first non-trivial eigenvalue of the (unnormalized) graph Laplacian over non-negative edge weights \eqref{e:EOP}. 
Our main result (\cref{t:UnitDistGraphReal}) is that the spectral realization of a graph using the eigenvectors corresponding to the solution to \eqref{e:EOP}, under certain assumptions, is a maximal graph realization.  We also prove a converse result (\cref{t:Converse}) that a maximal graph realization, under certain assumptions, has coordinate vectors which is are eigenvectors for \emph{some} graph Laplacian. In \cref{s:NumExamples}, we illustrated  \cref{t:Converse} and \cref{t:UnitDistGraphReal} with a number of examples. 

While the analysis here provides an interesting theoretical tool for finding unit-distance graph realizations, 
practically speaking, there are limitations since, before solving \eqref{e:EOP}, we do not know 
(i) the dimension $d = \dim E_{\lambda^\star}$ of the maximal graph realization or
(ii) if the assumption that  $w^\star > 0$ holds.
However, as shown in \cref{s:NumExamples}, there are a number of interesting graphs for which the dimension is either 2 or 3 and $w^\star > 0$ (see \cref{f:graphZoo}). 

There are a number of interesting extensions of this work. 
For structured graphs, it might be possible to explicitly identify a maximal graph realization, so this could yield interesting examples to test computational methods for eigenvalue optimization. 
In this paper, we considered the unnormalized graph Laplacian and optimized over nonnegative edge weights. It is possible to consider other matrices associated with a graph (\eg, the normalized graph Laplacian) and also vertex weights. Another variation would be to consider the spectral realization corresponding to an extremal first (non-trivial) Steklov eigenvalue for a vertex subset. 
Finally, in data analysis, pairwise comparison data is often represented using a $w$-weighted graph where the weights correspond to the pairwise comparisons. The spectral realization of a $G$ (for these fixed weights) is then used for analysis or visualization. It is an interesting question of whether properties of this spectral graph realization could be bounded by further analyzing the extremal graph realizations.

\subsection*{Acknowledgments}
B. Osting would like to thank  \'Edouard Oudet for useful conversations and the Laboratoire Jean Kuntzmann (LJK), Universit\'e Grenoble Alpes for hosting him during his sabbatical leave, where this work was initiated. 

\clearpage
\printbibliography

\end{document}

%% file: tikz/Illustration1.tex
\begin{tikzpicture}[thick,scale=1]
\tikzset{minimum size=6pt}

\draw (0,0) node[circle,fill,inner sep=1pt] (A){};
\draw (1,0) node[circle,fill,inner sep=1pt] (B){};
\draw (.5,0.866) node[circle,fill,inner sep=1pt] (C){};
\draw (-.5,.866) node[circle,fill,inner sep=1pt] (D){};
\draw (-1,0) node[circle,fill,inner sep=1pt] (E){};
\draw (-.5,-.866) node[circle,fill,inner sep=1pt] (F){};
\draw (.5,-.866) node[circle,fill,inner sep=1pt] (G){};
\draw (.25,.433)  node[red,circle,fill,inner sep=1pt] (H){};
          
 \draw[->] (A) -- (C);   
 \draw[dashed] (B) -- (D);
 \draw[->] (A) -- (B);
 \draw[->] (A) -- (D);
 \draw[blue] (B) -- (C);
 \draw[blue] (C) -- (D);
 \draw[blue] (D) -- (E);
 \draw[blue] (E) -- (F);
 \draw[blue] (F) -- (G);
 \draw[blue] (G) -- (B);
 \draw[thick,red,->] (A) -- (H);
        
\end{tikzpicture}

%% file: tikz/Illustration2.tex
\begin{tikzpicture}[thick,scale=1]
\tikzset{minimum size=6pt}

\draw (1,0) node[circle,fill,inner sep=1pt] (B){};
\draw (.5,0.866) node[circle,fill,inner sep=1pt] (C){};
\draw (-.5,.866) node[circle,fill,inner sep=1pt] (D){};
\draw (-1,0) node[circle,fill,inner sep=1pt] (E){};
\draw (-.5,-.866) node[circle,fill,inner sep=1pt] (F){};
\draw (.5,-.866) node[circle,fill,inner sep=1pt] (G){};
          
 \draw[blue] (B) -- (C);
 \draw[blue] (C) -- (D);
 \draw[blue] (D) -- (E);
 \draw[blue] (E) -- (F);
 \draw[blue] (F) -- (G);
 \draw[blue] (G) -- (B);

\draw (1.2,0) node[circle,fill,inner sep=1pt] (B1){};
\draw (.5,0.666) node[circle,fill,inner sep=1pt] (C1){};
\draw (-.5,.666) node[circle,fill,inner sep=1pt] (D1){};
\draw (-1.2,0) node[circle,fill,inner sep=1pt] (E1){};
\draw (-.5,-.666) node[circle,fill,inner sep=1pt] (F1){};
\draw (.5,-.666) node[circle,fill,inner sep=1pt] (G1){};
          
 \draw[red] (B1) -- (C1);
 \draw[red] (C1) -- (D1);
 \draw[red] (D1) -- (E1);
 \draw[red] (E1) -- (F1);
 \draw[red] (F1) -- (G1);
 \draw[red] (G1) -- (B1);

%
        
\end{tikzpicture}

%% file: tikz/Illustration3.tex
\begin{tikzpicture}[thick,scale=1]
\tikzset{minimum size=6pt}
\graph[nodes={draw, circle,fill,inner sep=1pt}, 
           empty nodes, branch down=.4 cm,
           grow right sep=2cm] {subgraph I_nm [V={a, b, c, d}, W={1,...,3}];
  a -- { 1,3};
  b -- { 1, 2 };
  c -- { 2,3 };
  d -- {2};
};        

\draw (0.1,-1.9) node[circle,fill,inner sep=1pt] (A){};
\draw (2.4,-1.9) node[circle,fill,inner sep=1pt] (B){};

 \draw[blue,->] (1.3,-1.3) -- (1.3,-1.8);
 \draw[black] (A) -- (B);
 
\end{tikzpicture}

%% file: tikz/Illustration4.tex
\begin{tikzpicture}[line join = round, line cap = round,thick]
\pgfmathsetmacro{\factor}{1/sqrt(2)};
\coordinate [] (A) at (1,0,-1*\factor);
\coordinate [] (B) at (-1,0,-1*\factor);
\coordinate [] (C) at (0,1,1*\factor);
\coordinate [] (D) at (0,-1,1*\factor);

\draw (A) node[circle,fill,inner sep=2pt] (){};
\draw (B) node[circle,fill,inner sep=2pt] (){};
\draw (C) node[circle,fill,inner sep=2pt] (){};
\draw (D) node[circle,fill,inner sep=2pt] (){};

\draw[] (A)--(D)--(B)--cycle;
\draw[-, fill=red!30, opacity=.2] (A)--(D)--(B)--cycle;

\draw[] (A) --(D)--(C)--cycle;
\draw[-, fill=green!30, opacity=.2] (A) --(D)--(C)--cycle;

\draw[] (B)--(D)--(C)--cycle;
\draw[-, fill=blue!30, opacity=.2] (B)--(D)--(C)--cycle;
\end{tikzpicture}